\documentclass{article}
\usepackage{arxiv}

\usepackage[T1]{fontenc}    
\usepackage{xcolor}
\usepackage{url}
\usepackage{amsmath, amsfonts, amssymb, mathtools, mathrsfs, extarrows, esint}
\usepackage{paralist}
\usepackage[numbers,sort]{natbib}

\numberwithin{equation}{section}

\usepackage{amsthm}
\usepackage{thmtools}
\newtheorem{theorem}{Theorem}[section]
\declaretheorem[sibling=theorem]{lemma}
\declaretheorem[sibling=theorem]{proposition}
\declaretheorem[sibling=theorem]{corollary}
\declaretheorem[sibling=theorem]{definition}
\declaretheorem[sibling=theorem]{remark}
\declaretheorem[sibling=theorem]{example}
\declaretheorem[sibling=theorem]{assumption}

\usepackage[colorlinks,linkcolor=violet,citecolor=blue,urlcolor=black]{hyperref}
\usepackage[capitalize,nameinlink,noabbrev]{cleveref}
\crefdefaultlabelformat{#2\textup{#1}#3} 
\Crefname{assumption}{Assumption}{Assumptions}

\usepackage{tikz}
\usetikzlibrary{graphs, quotes, positioning, shapes.geometric, arrows.meta, fit, positioning, shapes.geometric, shapes.arrows, decorations.markings}
\usepackage{tikz-3dplot}

\newcommand{\R}{\mathbb{R}}
\newcommand{\curl}{\nabla\times}
\renewcommand{\div}{\nabla\cdot}
\newcommand{\Hc}[1]{H({\rm curl},#1)}
\newcommand{\Hd}[1]{H({\rm div},#1)}
\DeclareMathOperator{\dist}{dist}
\DeclareMathOperator{\diam}{diam}
\DeclareMathOperator{\sgn}{sgn}
\DeclareMathOperator{\Area}{Area}

\title{Magneto-Acousto-Electric Tomography \\ with Magnetic Field Measurements: \\ Modeling, Inversion and Stability}
\author{Lingyun Qiu$^{1,2}$, Siqin Zheng$^{3}$ \\
	$^1$ Yau Mathematical Sciences Center, Tsinghua University, Beijing 100084, China\\
    $^2$ Yanqi Lake Beijing Institute of Mathematical Sciences and Applications, Beijing 101408, China\\
    (\href{mailto:lyqiu@tsinghua.edu.cn}{lyqiu@tsinghua.edu.cn})\\
	$^3$ Department of Mathematical Sciences, Tsinghua University, Beijing 100084, China \\
	(\href{mailto:zhengsq21@mails.tsinghua.edu.cn}{zhengsq21@mails.tsinghua.edu.cn}) \\
}

\renewcommand{\shorttitle}{MAET with Magnetic Field Measurements}

\begin{document}

\maketitle

\begin{abstract}
    Magneto-acousto-electric tomography (MAET) combines ultrasound with a static magnetic field to infer the electrical conductivity of an object.
    In this paper, we present a rigorous quasi-static mathematical model for MAET with magnetic field measurements and introduce an adjoint problem to decouple the resulting hybrid inverse problem. 
    This yields a two-step inversion procedure: solving an acoustic inverse source problem and then recovering the conductivity from an internal current density.
    For the second step, by exploiting the analytic structure of a coil-determined field, we establish an interior H\"older stability estimate without imposing a pointwise nonzero constraint on the internal data.
    We further prove that, under explicit smallness assumptions on the conductivity and coil geometry, the conductivity can be recovered with region-of-interest Lipschitz stability in both bounded and half-space geometries.
\end{abstract}



\section{Introduction}

In clinical diagnosis and pathological research, electrical conductivity is a crucial physical property of biological tissues for distinguishing between normal and pathological tissues. Therefore, conductivity imaging plays a critical role in medical imaging. 

In mathematics, an imaging problem is typically modeled as an inverse problem, and the imaging resolution depends on the stability of the corresponding inverse problem. The conventional method, electrical impedance tomography (EIT), has low resolution because its inverse problem exhibits only logarithmic stability \cite{Uhlmann2009}. This means that errors in data are significantly amplified during the inversion procedure, resulting in highly inaccurate reconstructions. To achieve more accurate conductivity inversion, various coupled-physics imaging methods have been developed in recent years. These methods combine the electrical process with other processes such as acoustics and magnetism, utilizing observable signals generated by their interaction, such as acousto-electric effect and electromagnetic induction, to reconstruct the conductivity. The inverse problems of these methods typically demonstrate better stability, often of a H\"older or even Lipschitz type, leading to higher resolution. For detailed introductions on coupled-physics imaging, we refer to \cite{Widlak2012} and the references therein.

Magneto-acousto-electric tomography (MAET) is a specific modality of coupled-physics imaging method designed based on the Hall effect, where vibrating charged particles in a magnetic field are subjected to the Lorentz force, thereby generating electromagnetic signals. When first proposed, MAET relied on measurements of the induced voltage on the boundary of the imaging object \cite{Wen1998,Haider2008,Roth2009}. Kunyansky develops a rigorous and general mathematical model and an almost explicit inversion procedure \cite{Kunyansky2012}. One of the key steps, recovering the conductivity from internal current densities, requires suitable design of boundary conditions to generate linearly independent current densities. A potential approach is explored in \cite{Qiu2023}, where local Lipschitz stability results are derived for both full data and partial data. Kunyansky et al. also consider other settings for better applications \cite{Kunyansky2017,Kunyansky2023}. As a noninvasive technique that does not involve ionizing radiation, MAET poses fewer risks than traditional X-ray imaging and thus holds promise for clinical applications.

More recently, measurements of the induced electromotive force (emf) in coils placed outside the object have been investigated \cite{Zengin2012,Guo2015,Zengin2016}. See \cref{fig:setup} for a simple experimental setup. This variant is referred to as MAET with magnetic field measurements, since by Faraday's law, the measured emf in a coil corresponds to the rate of change of the magnetic flux through its bounding surface. Compared to the boundary measurements using electrodes, this contactless approach avoids the problems caused by contact impedances and the screening effect of superficial insulating layers such as skin and bone \cite{Tarjan1968,Gencer1999}. It also eliminates patient discomfort in clinical diagnosis.

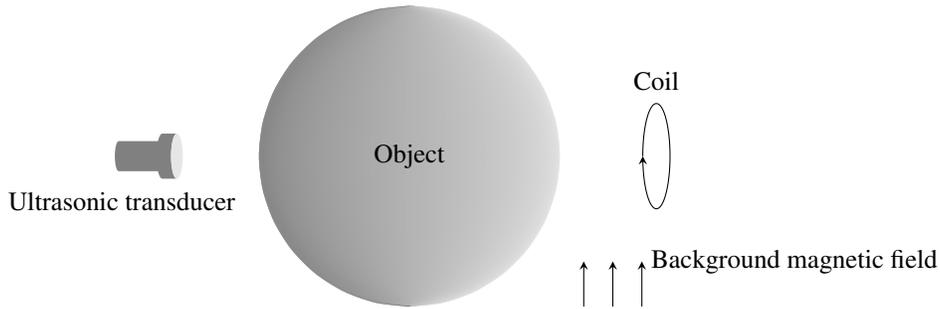
\begin{figure}[htbp]
    \centering
    \tdplotsetmaincoords{90}{105}
    \begin{tikzpicture}[tdplot_main_coords, scale=2, >=stealth]
        \foreach \phi in {105,106,...,283} {
            \pgfmathsetmacro\shade{10+(\phi-105)*4/18}
            \fill[black!\shade] 
                plot[smooth, domain=0:180, samples=15, variable=\theta] 
                ({sin(\theta)*cos(\phi)}, {sin(\theta)*sin(\phi)}, {cos(\theta)}) --
                plot[smooth, domain=180:0, samples=15, variable=\theta] 
                ({sin(\theta)*cos(\phi+2)}, {sin(\theta)*sin(\phi+2)}, {cos(\theta)}) --
                cycle;
        }
        \node at (0,0,0) {Object};
        \draw[->] plot[smooth, domain=0:360, samples=60, variable=\theta]
            ({0.35*cos(\theta)}, 1.7, {0.35*sin(\theta)});
        \node at (0,1.7,0.5) {Coil};
        \pgfmathsetmacro\sr{0.1}
        \fill[black!50]
            plot[smooth, domain=0:180, samples=30, variable=\theta]
            ({\sr*sin(\theta)}, -2, {\sr*cos(\theta)}) --
            plot[smooth, domain=180:0, samples=30, variable=\theta]
            ({\sr*sin(\theta)}, -1.7, {\sr*cos(\theta)}) --
            cycle;
        \pgfmathsetmacro\lr{0.15}
        \fill[black!50]
            plot[smooth, domain=0:180, samples=30, variable=\theta]
            ({\lr*sin(\theta)}, -1.7, {\lr*cos(\theta)}) --
            plot[smooth, domain=180:0, samples=30, variable=\theta]
            ({\lr*sin(\theta)}, -1.6, {\lr*cos(\theta)}) --
            cycle;
        \fill[black!10]
            plot[smooth, domain=0:360, samples=60, variable=\theta]
            ({\lr*sin(\theta)}, -1.6, {\lr*cos(\theta)});
        \node at (-1.2,-2.3,-0.3) {Ultrasonic transducer};
        \draw[->] (0,1.2,-1) -- (0,1.2,-0.7);
        \draw[->] (0,1.4,-1) -- (0,1.4,-0.7);
        \draw[->] (0,1.6,-1) -- (0,1.6,-0.7) node[anchor=west]{Background magnetic field};
    \end{tikzpicture}
    \caption{A simple experimental setup of MAET with magnetic field measurements}
    \label{fig:setup}
\end{figure}

Existing models of MAET with magnetic field measurements are based on the Maxwell's system under simplifying approximations, such as neglecting displacement currents and assuming low conductivity, which lack rigorous mathematical justification. These simplifications, while useful for initial experimental validation \cite{Kaboutari2019,Tetik2023,Gozu2024}, restrict the imaging resolution, reliability, and clinical applicability. Therefore, establishing a rigorous mathematical foundation is essential to advance this technique by properly decoupling the interacting physics and guiding experimental designs for stable reconstruction. This work aims to fill this gap by developing a rigorous model and inversion framework, analogous to MAET with voltage measurements \cite{Kunyansky2012}. Additionally, we analyze the stability of conductivity reconstruction, offering insights into appropriate configurations of detection coils.

The inverse problem of a coupled-physics imaging technique can typically be decoupled into two steps: first, reconstructing some internal functionals or vector fields from boundary or external measurements; second, using these to recover the unknown coefficients. Further examples and details can be found in \cite{Bal2013a,Alberti2018} and the references therein. The inversion procedure of MAET with voltage measurements, as developed in \cite{Kunyansky2012}, conforms to this paradigm. 

Regarding MAET with magnetic field measurements, since the underlying principle is similar, we aim to follow this framework. We adopt a quasi-static model of the Hall effect instead of the full Maxwell's system, because the electromagnetic induction happens much more rapidly than ultrasonic propagation. This significant difference in time scales plays a crucial role in the decoupling process and inversion procedure. A stationary adjoint equation is introduced, enabling the decoupling of the hybrid inverse problem into an inverse source problem for the acoustic wave equation and an inverse medium problem for the electrostatic equation with internal data. The first problem is analogous to that in MAET with voltage measurements, whereas the second does not admit an explicit reconstruction formula. 
We first use a multiplier approach to derive a weighted stability estimate. By combining it with a polynomial lower bound for the local $L^2$-norm of the adjoint current and applying a weighted interpolation argument, we establish an interior H\"older stability estimate without imposing a pointwise nonzero constraint on the current. Under stronger smallness conditions on the conductivity and coil geometry, we further obtain Lipschitz stability in regions of interest where the adjoint current is bounded away from zero.

A notable feature of the stability results is that they do not require any a priori knowledge of the boundary values of the conductivity. This property originates from the structure of the magnetic measurement model and contrasts with many inverse conductivity problems where boundary information is incorporated into the reconstruction process, including EIT \cite{Uhlmann2009} and MAET with voltage measurements \cite{Kunyansky2012,Qiu2023}. This distinction highlights a fundamental difference between the magnetic field and voltage measurement modalities of MAET. 
Moreover, the H\"older stability result reveals a structural advantage of the coil-based measurement modality. Pointwise nonzero constraints on PDE solutions or internal data are a recurring issue in the second reconstruction step of hybrid inverse problems and are often enforced through carefully chosen deterministic or randomized boundary data \cite{Alberti2018,Alberti2022a}. The weighted interpolation argument has been used to derive conditional stability in the presence of zeros \cite{Choulli2019a,Choulli2021,Bonnetier2022}. In the present setting, the required polynomial lower bound follows from the analyticity and intrinsic nontriviality of a coil-determined field. This allows the pointwise nonzero constraint to be removed for the interior H\"older estimate. The stronger region-of-interest Lipschitz estimate still requires a positive lower bound, which motivates our sufficient conditions based on the coil geometry and prior information on the conductivity.

This paper is structured as follows. \Cref{sec:model} introduces the quasi-static model of MAET with magnetic field measurements and relevant preliminaries. \Cref{sec:inv} develops the inversion procedure. \Cref{sec:stab} addresses the key challenge of stabilizing conductivity reconstruction from an internal current density, providing conditional stability results for both bounded domains and half-space. Concluding remarks are presented in \cref{sec:con}.

\section{Mathematical model}\label{sec:model}

Since the time scales of the ultrasonic and electromagnetic processes differ by orders of magnitude \cite{Ammari2015}, we model the forward problem of MAET in a quasi-static manner instead of using the full Maxwell's equations. Specifically, at any given instant during the acoustic wave propagation, the electromagnetic induction in the space can be regarded as nearly instantaneous, and therefore the overall process can be effectively treated using a quasi-static approximation as in \cite{Kunyansky2012}.

Denoting the imaging object by a domain $\Omega$ in $\R^3$, MAET aims to reconstruct the conductivity $\sigma(x)$, which is a spatial function in $\Omega$. In the MAET experiments, the object is placed in a static magnetic field with constant magnetic induction intensity $B_0$. When an acoustic wave is transmitted, it propagates through the object and causes the charged particles inside to vibrate with velocity $V(x,t)$. The arising Lorentz force leads to the flow of charged particles, i.e., the Lorentz current $\sigma(x) V(x,t)\times B_0$. At each time instant $t$, an electric potential $u(x,t)$ is instantaneously induced within the object, with a secondary current $-\sigma(x)\nabla u(x,t)$. Due to charge conservation and the instantaneous redistribution of charges within the imaging object, the total current density $J = \sigma V\times B_0 - \sigma\nabla u$ is divergence-free. Furthermore, since the object is typically placed in water or saline of low conductivity, which can be considered insulating, the total current $J$ does not flow out through the boundary. These physical considerations lead to the following Neumann problem for the potential $u$:
\begin{equation}\label{eq:origin}
    \begin{cases}
        \div(\sigma\nabla u) = \div(\sigma V\times B_0) & \text{in}\ \Omega, \\
        \sigma\partial_\nu u = \sigma V\times B_0\cdot\nu & \text{on}\ \partial\Omega,
    \end{cases}
\end{equation}
where $\nu$ denotes the unit outer normal direction on the boundary $\partial\Omega$. According to the Biot--Savart law of magnetostatics, the total current $J$ generates a magnetic field with flux density
\begin{equation}\label{eq:B}
    B(y,t) = \frac{\mu}{4\pi}\int_\Omega J(x,t)\times\frac{y-x}{|y-x|^3}\,dx
\end{equation}
in the whole space $\R^3$, where the permeability $\mu$ is assumed as the constant of vacuum. 

In the present work, the measuring coil is idealized as a thin conducting wire represented by an oriented simple closed curve $L\subset \R^3\setminus \overline{\Omega}$. Accordingly, the magnetic flux is computed through an oriented surface $D$ satisfying $\partial D=L$, with the orientation chosen consistently with that of $L$ by the right-hand rule. This corresponds to the standard thin-wire approximation, which is appropriate when the coil has one turn, or a small number of turns, and its cross-sectional size is negligible compared with the distance to the imaging region. Under this setting, the measured time-varying induced emf in $L$ is, according to Faraday's law of induction, the negative rate of change of the magnetic flux:
\begin{equation}\label{eq:m_L}
    \begin{aligned}
        m_L(t) &= -\frac{d}{dt}\int_D B(y,t)\cdot n(y)\,dS_y \\
        &= -\frac{\mu}{4\pi}\int_\Omega \partial_t J(x,t)\cdot \int_D \frac{(y-x)\times n(y)}{|y-x|^3}\,dS_y\,dx.
    \end{aligned}
\end{equation}
Here $n$ denotes the unit normal vector to the surface $D$.

The aim of MAET with magnetic field measurements is to reconstruct the conductivity $\sigma$ in $\Omega$, from $m_L(t)$ for several coils $L$, under different background field $B_0$ and ultrasonic excitation $V(x,t)$.

Denoting
\begin{equation}\label{eq:C_L}
    C_L(x) = -\frac{\mu}{4\pi}\int_D \frac{(y-x)\times n(y)}{|y-x|^3}\,dS_y = -\frac{\mu}{4\pi}\curl\int_D \frac{n(y)}{|y-x|}\,dS_y,
\end{equation}
which totally depends on the coil $L$ and is time-independent, the measurement \eqref{eq:m_L} can be written as
\begin{equation}\label{eq:m_L2}
    m_L = \int_\Omega \sigma(\partial_t V\times B_0 - \nabla v)\cdot C_L\,dx,
\end{equation}
where $v = \partial_t u$ solves the Neumann problem
\begin{equation}\label{eq:v}
    \begin{cases}
        \div(\sigma\nabla v) = \div(\sigma\partial_t V\times B_0) & \text{in}\ \Omega, \\
        \sigma\partial_\nu v = \sigma\partial_t V\times B_0\cdot\nu & \text{on}\ \partial\Omega.
    \end{cases}
\end{equation}

The formula \eqref{eq:m_L2} shows the coupling effect between acoustic and electromagnetic processes. Indeed, the measured signal $m_L$ depends on the acoustic field $V$ and the unknown conductivity $\sigma$, while the potential $v$ appearing in the expression is also determined by both of them through the Neumann problem \eqref{eq:v}. The role of the following adjoint formulation is to decouple this effect. More precisely, we shall rewrite $m_L$ as an acoustic measurement whose spatial source is determined by a time-independent adjoint current density. Once this source is recovered from the acoustic inverse source problem, the remaining step is to reconstruct the conductivity from the adjoint current through an overdetermined system. Thus the adjoint problem introduced below provides a significant bridge between the physical measurement model and the decoupling inversion procedure, which will be derived in \cref{sec:inv} and summarized in \cref{fig}.

Now we introduce the following adjoint problem
\begin{equation}\label{eq:w}
    \begin{cases}
        \div(\sigma\nabla w_L) = \div(\sigma C_L) & \text{in}\ \Omega, \\
        \sigma\partial_\nu w_L = \sigma C_L\cdot\nu & \text{on}\ \partial\Omega,
    \end{cases}
\end{equation}
so that the second term in \eqref{eq:m_L2} can be rewritten by integration by parts into
\begin{equation*}
    \int_\Omega \sigma\nabla v\cdot C_L\,dx = \int_\Omega \sigma\nabla v\cdot\nabla w_L\,dx = \int_\Omega \sigma\partial_t V\times B_0\cdot\nabla w_L\,dx.
\end{equation*}
Therefore,
\begin{equation}\label{eq:data}
    m_L = \int_\Omega \sigma\partial_t V\times B_0 \cdot(C_L - \nabla w_L)\,dx = \frac{B_0}{\rho}\cdot\int_\Omega J_L\times \nabla p\,dx,
\end{equation}
where $J_L = \sigma(C_L - \nabla w_L)$ is the total current density of the adjoint problem \eqref{eq:w} (called the adjoint current for short in the following), and $p$ is the acoustic pressure satisfying
\begin{equation*}
    (\partial_t^2 - c^2\Delta)p = 0, \quad \partial_t V = \frac{1}{\rho}\nabla p
\end{equation*}
with sound speed $c$ and mass density $\rho$ regarded as known constants.

\begin{remark}
    By the standard theory of elliptic equations, the potentials $v$ in \eqref{eq:v} and $w_L$ in \eqref{eq:w} are determined only up to additive constants. This grounding ambiguity is irrelevant for the measured data $m_L$ and for the subsequent reconstruction procedure, since only the gradients $\nabla v$ and $\nabla w_L$ enter the formula \eqref{eq:data} for $m_L$, the adjoint current $J_L = \sigma(C_L-\nabla w_L)$, and the decoupled expression below. When needed, one may fix the additive constants by imposing, for instance, zero mean over $\Omega$.
\end{remark}

\begin{remark}
    The thin-wire assumption is made for simplicity of presentation and is not essential to the decoupling argument above. More general coil geometries can be incorporated by superposition. For instance, if the measuring coil is modeled as a cylindrical winding with finite height, it can be regarded as a family of closed curves $L = \{L_s: s\in I\}$, or equivalently a continuous distribution of turns, along the height of the cylinder. The measured signal then takes the form
    \begin{equation*}
        \tilde{m}_L(t) = \int_I \eta(s)m_{L_s}(t)\,ds,
    \end{equation*}
    where $\eta$ denotes the turn density or weight function. By linearity, this additional integration can be absorbed into the coil-dependent vector fields by replacing $C_L$ with
    \begin{equation*}
        \tilde{C}_L(x) = \int_I \eta(s)C_{L_s}(x)\,ds.
    \end{equation*}
    With this replacement, the adjoint problem, the expression for the measured data, and the decoupling procedure remain exactly the same. Therefore, throughout the paper we work with a single closed curve to keep the notation simple.
\end{remark}

Before investigating how to recover the conductivity $\sigma$ from these measurements, we introduce some necessary notations and preliminaries on the adjoint problem \eqref{eq:w}.

\subsection{Notations}

During estimates, we use the symbol $C$ to denote positive constants with their dependence in parentheses.

For $k\in\mathbb{N}$ and a domain $\Omega\subset\R^d$, we use the standard notation $C^k(\overline\Omega)$ to denote the space of $k$-times continuously differentiable functions on $\overline{\Omega}$, with the norm
\begin{equation*} 
    \|u\|_{C^k(\overline\Omega)} = \max_{|\beta|\le k} \sup_{x\in \overline{\Omega}} |\partial^\beta u(x)|. 
\end{equation*}
For $\alpha\in(0,1]$, we denote by $C^{k,\alpha}(\overline{\Omega})$ the H\"older space consisting of $u\in C^k(\overline{\Omega})$ whose $k$-th partial derivatives are H\"older continuous with exponent $\alpha$, that is,
\begin{equation*} 
    [\partial^k u]_{\alpha;\Omega} = \max_{|\beta| = k} \sup_{\substack{x,y\in\overline{\Omega} \cr x\ne y}} \frac{|\partial^\beta u(x) - \partial^\beta u(y)|}{|x-y|^\alpha} < \infty, 
\end{equation*}
with the norm
\begin{equation*} 
    \|u\|_{C^{k,\alpha}(\overline{\Omega})} = \|u\|_{C^k(\overline{\Omega})} + [\partial^k u]_{\alpha;\Omega}. 
\end{equation*}
We say $\partial\Omega\in C^{k,\alpha}$ if each point of $\partial\Omega$ has a neighborhood in which $\partial\Omega$ is the graph of a $C^{k,\alpha}$ function of $d-1$ of the coordinates.
We use $C_c^\infty(\Omega)$ to denote the space of smooth functions with compact supports in $\Omega$.

For $p\in[1,\infty]$, we use the notation $W^{k,p}(\Omega)$ for the $L^p$-based Sobolev space of $k$-th order, i.e.
\begin{equation*} 
    W^{k,p}(\Omega) = \{u\in L^p(\Omega): \partial^\beta u\in L^p(\Omega),\ \forall\,|\beta|\le k\},
\end{equation*}
with the norm
\begin{equation*} 
    \|u\|_{W^{k,p}(\Omega)} = \sum_{|\beta|\le k} \|\partial^\beta u\|_{L^p(\Omega)}.
\end{equation*}
When $p=2$, the space $W^{k,2}(\Omega)$ is usually abbreviated as $H^k(\Omega)$.
In general, we introduce the fractional Sobolev space
\begin{equation*}
    H^s(\R^d) = \{u\in L^2(\R^d): (1+|\xi|^2)^{s/2}\hat{u}(\xi)\in L^2(\R^d)\} 
\end{equation*}
for $s>0$, with the norm
\begin{equation*} 
    \|u\|_{H^s(\R^d)} = \|(1+|\xi|^2)^{s/2}\hat{u}\|_{L^2(\R^d)}, 
\end{equation*}
where $\hat{u}$ is the Fourier transform of $u$.
When the boundary $\partial\Omega\in C^{0,1}$ (usually called Lipschitz), the space $H^s(\partial\Omega)$ is defined via local charts for $s\in(0,1)$. In this paper, we only need the space $H^{1/2}(\partial\Omega)$, which is actually the space of traces of $H^1(\Omega)$ and the trace theorem shows that
\begin{equation}\label{eq:tr}
    \|u\|_{H^{1/2}(\partial\Omega)} \le C(d,\Omega)\|u\|_{H^1(\Omega)}
\end{equation}
Its dual space is defined as $H^{-1/2}(\partial\Omega)$, and the duality pair between them is denoted by $\langle\cdot,\cdot\rangle_{\partial\Omega}$.
For more details about Sobolev spaces, see \cite{Adams2003} and \cite[Chapter 5]{Evans2010}. For simplicity, we do not distinguish the notation between (scalar) functions (usually denoted by lowercase letters) and vector fields (usually denoted by uppercase letters).

For $\Omega\subset\R^3$, we define the curl and divergence operators on the space of $L^2$ vector fields in the distributional sense, i.e.
\begin{align*}
    \langle \curl F, \Phi\rangle &= \int_\Omega F\cdot\curl\Phi\,dx, \quad \forall\, \Phi\in C_c^\infty(\Omega), \\
    \langle \div F, \phi\rangle &= -\int_\Omega F\cdot\nabla\phi\,dx, \quad \forall\, \phi\in C_c^\infty(\Omega).
\end{align*}
They induce the spaces
\begin{align*}
    \Hc{\Omega} &= \{F\in L^2(\Omega): \curl F\in L^2(\Omega)\}, \\
    \Hd{\Omega} &= \{F\in L^2(\Omega): \div F\in L^2(\Omega)\}
\end{align*}
with the norms
\begin{align*}
    \|F\|_{\Hc{\Omega}} &= \|F\|_{L^2(\Omega)} + \|\curl F\|_{L^2(\Omega)}, \\
    \|F\|_{\Hd{\Omega}} &= \|F\|_{L^2(\Omega)} + \|\div F\|_{L^2(\Omega)}.
\end{align*}
Through integration by parts and density arguments, we can define the tangential trace $F\times\nu\in H^{-1/2}(\partial\Omega)$ for $F\in\Hc{\Omega}$ by
\begin{equation*}
    \langle F\times\nu, \Phi\rangle_{\partial\Omega} = \int_\Omega F\cdot\curl\Phi\,dx - \int_\Omega \curl F\cdot\Phi\,dx, \quad \forall\,\Phi\in H^1(\Omega),
\end{equation*}
and the normal trace $F\cdot\nu\in H^{-1/2}(\partial\Omega)$ for $F\in\Hd{\Omega}$ by
\begin{equation*}
    \langle F\cdot\nu, \phi\rangle_{\partial\Omega} = \int_\Omega F\cdot\nabla\phi\,dx + \int_\Omega (\div F)\phi\,dx, \quad \forall\,\phi\in H^1(\Omega).
\end{equation*}
Using these definitions and the trace theorem \eqref{eq:tr}, we know that
\begin{align}
    \|F\times\nu\|_{H^{-1/2}(\partial\Omega)} &\le C(\Omega)\|F\|_{\Hc{\Omega}} \label{eq:tan_tr} \\
    \|F\cdot\nu\|_{H^{-1/2}(\partial\Omega)} &\le C(\Omega)\|F\|_{\Hd{\Omega}} \label{eq:nor_tr}
\end{align}
Further details on these spaces of vector fields can be found in \cite{Amrouche1998,Amrouche2013}.

For notational convenience, we denote the admissible set of electrical conductivities by
\begin{align*}
    \Sigma(\lambda) &= \{\sigma\in L^\infty(\Omega): \lambda\le\sigma\le\lambda^{-1}\ \text{in}\ \Omega\}, \\
    \Sigma(\lambda,\Lambda) &= \{\sigma\in \Sigma(\lambda)\cap W^{1,\infty}(\Omega): \|\sigma\|_{W^{1,\infty}(\Omega)}\le\Lambda\}
\end{align*}
with constants $\lambda\in(0,1)$ and $\Lambda>1$.

\subsection{Preliminaries on the adjoint problem}

\begin{definition}
    We say that $w_L\in H^1(\Omega)$ is a weak solution to the problem \eqref{eq:w}, if for any $\phi\in H^1(\Omega)$,
    \begin{equation}\label{eq:weak}
        \int_\Omega \sigma\nabla w_L\cdot\nabla\phi\,dx = \int_\Omega \sigma C_L\cdot\nabla\phi\,dx.
    \end{equation}
\end{definition}

It is obvious by the definition \eqref{eq:C_L} that $C_L$ is smooth in $\Omega$ once $L\subset\R^3\setminus\overline{\Omega}$. Hence, when considering the regularity of the solution $w_L$, only the regularity of the conductivity $\sigma$ and the domain $\Omega$ matter.

\begin{proposition}\label{prop:forward}
    Let $\Omega$ be a bounded domain in $\R^3$. If $\sigma\in \Sigma(\lambda)$, then
    \begin{enumerate}
        \item the problem \eqref{eq:w} admits a weak solution, and it is unique up to an additive constant and satisfies
        \begin{equation}\label{eq:grad}
            \|\nabla w_L\|_{L^2(\Omega)} \le \lambda^{-2}\|C_L\|_{L^2(\Omega)};
        \end{equation}
        \item the total current $J_L = \sigma(C_L - \nabla w_L)\in L^2(\Omega)$ and
        \begin{equation}
            \|J_L\|_{L^2(\Omega)} \le C(\lambda)\|C_L\|_{L^2(\Omega)};
        \end{equation}
        \item if additionally $\sigma\in W^{1,\infty}(\Omega)$, then $J_L\in\Hc{\Omega}$.
    \end{enumerate}
\end{proposition}
\begin{proof}
    The existence and uniqueness up to an additive constant is a standard result by the Poincar\'e inequality
    \begin{equation}\label{eq:Poincare}
        \left\|u - \fint_\Omega u\,dx\right\|_{L^2(\Omega)} \le C(\Omega)\|\nabla u\|_{L^2(\Omega)}, \quad \forall\,u\in H^1(\Omega)
    \end{equation}
    and the Lax--Milgram theorem. Taking $\phi = w_L$ in \eqref{eq:weak} leads to
    \begin{equation*}
        \lambda\|\nabla w_L\|_{L^2(\Omega)}^2 \le \lambda^{-1}\|C_L\|_{L^2(\Omega)}\|\nabla w_L\|_{L^2(\Omega)},
    \end{equation*}
    so the gradient estimate \eqref{eq:grad} holds. Therefore, $J_L = \sigma(C_L - \nabla w_L)\in L^2(\Omega)$ and
    \begin{equation*}
        \|J_L\|_{L^2(\Omega)} \le \lambda^{-1}(\|C_L\|_{L^2(\Omega)} + \|\nabla w_L\|_{L^2(\Omega)}) \le C(\lambda)\|C_L\|_{L^2(\Omega)}.
    \end{equation*}
    Finally, if $\sigma\in W^{1,\infty}(\Omega)$, then by direct computation we see that
    \begin{equation*}
        \curl J_L = \nabla\sigma\times(C_L - \nabla w_L) + \sigma\curl C_L \in L^2(\Omega),
    \end{equation*}
    yielding that $J_L\in\Hc{\Omega}$.
\end{proof}

We also need the following lemma \cite[Theorem 3.17]{Amrouche1998}, showing that a solenoidal vector field $F$ which is tangential on the boundary admits a normal vector potential $A$.
\begin{assumption}\label{as}
    Let $\Omega$ be a bounded, simply connected domain in $\R^3$ with a connected boundary of class $C^{0,1}$.
\end{assumption}
\begin{lemma}\label{lem:vec_pot}
    Let $\Omega$ satisfy \cref{as}. Then for any vector field $F\in\Hd{\Omega}$ satisfying
    \begin{equation}\label{eq:F}
        \begin{cases}
            \div F = 0 & \text{in}\ \Omega, \\ 
            F\cdot\nu = 0 & \text{on}\ \partial\Omega,
        \end{cases}
    \end{equation}
    there exists a vector field $A\in\Hc{\Omega}$ satisfying
    \begin{equation}\label{eq:A}
        \begin{cases}
            \div A = 0 & \text{in}\ \Omega, \\
            A\times\nu = 0 & \text{on}\ \partial\Omega,
        \end{cases}
    \end{equation}
    such that $F = \curl A$ in $\Omega$, and
    \begin{equation}
        \|A\|_{L^2(\Omega)} \le C(\Omega)\|F\|_{L^2(\Omega)}.
    \end{equation}
\end{lemma}

Since the equation \eqref{eq:w} means that the total current $J_L$ satisfies the equation \eqref{eq:F}, we have
\begin{corollary}\label{cor:J_A}
    Under \cref{as}, if $\sigma\in\Sigma(\lambda)$, then the total current $J_L$ of the problem \eqref{eq:w} admits a vector potential $A\in\Hc{\Omega}$, satisfying \eqref{eq:A} and
    \begin{equation}\label{eq:est_A}
        \|A\|_{L^2(\Omega)} \le C(\lambda,\Omega)\|C_L\|_{L^2(\Omega)}.
    \end{equation}
\end{corollary}

\section{Inversion procedure}\label{sec:inv}

\subsection{Recovering the total effect of an acoustic source}

In the first step of reconstruction, we consider recovering some information inside $\Omega$ from the measured data \eqref{eq:data}. To this end, we proceed as in MAET with voltage measurements \cite{Kunyansky2012} to employ the spherical wave
\begin{equation*}
    p(x,t;y) = \frac{\delta(|x-y|-ct)}{4\pi c^2 t}
\end{equation*}
transmitted at a point $y$ outside the object, which is generated by the initial state
\begin{equation*}
    p(x,0;y) = 0, \quad \partial_t p(x,0;y) = \delta(x-y),
\end{equation*}
so that we can express the measured data as a convolution in space
\begin{align*}
    m_L(t;y) &= \frac{B_0}{\rho}\cdot \int_\Omega J_L(x)\times\nabla_x p(x,t;y) \,dx \\
    &= \left<\frac{B_0}{\rho}\cdot[J_L\times\nu\delta_{\partial\Omega} + (\curl J_L)\chi_\Omega], p(\cdot,t;y)\right>.
\end{align*}
Here the term $J_L\times\nu\delta_{\partial\Omega}$ should be understood in the sense of distributions. Its appearance comes from the integration by parts formula for the curl when the current $J_L$, originally defined only in $\Omega$, is regarded as a vector field in the whole space via zero extension. More precisely, for any test function $\phi\in C^\infty(\R^3)$, we have
\begin{equation*}
    \int_\Omega J_L\times\nabla\phi\,dx = \int_\Omega \phi \curl J_L\,dx + \int_{\partial\Omega} \phi J_L\times\nu\,dS.
\end{equation*}
Thus, in the distributional sense,
\begin{equation*}
    \curl(J_L\chi_\Omega) = (\curl J_L)\chi_\Omega + J_L\times\nu\delta_{\partial\Omega},
\end{equation*}
containing both the interior curl of the adjoint current and its tangential trace on the boundary.

Let
\begin{equation*}
    f = \frac{B_0}{\rho}\cdot[J_L\times\nu\delta_{\partial\Omega} + (\curl J_L)\chi_\Omega].
\end{equation*}
Then the measurement
\begin{equation*}
    m_L(t;y) = \frac{t}{4\pi} \int_{|z|=1} f(y+ctz)\,dS_z
\end{equation*}
as a function of $t>0$ and $y\in\R^3$, satisfies an acoustic wave equation
\begin{equation}\label{eq:meas_acoustic}
    \begin{cases}
        (\partial_t^2 - c^2\Delta_y)m_L(t;y) = 0, & y\in\R^3,\ t>0, \\
        m_L(0;y) = 0, & y\in\R^3, \\
        \partial_t m_L(0;y) = f(y), & y\in\R^3.
    \end{cases}
\end{equation}
So far we have decoupled the original coupled-physics inverse problem into two parts: an acoustic part and an electromagnetic part. The acoustic part involves an inverse source problem of the acoustic wave equation, to recover the initial state $f(y)$ (which can be regarded as an instantaneous source at the initial time), which is a spatial distribution supported on $\overline{\Omega}$. It also acts as the first inversion step of many other coupled-physics imaging techniques, such as thermo-acoustic tomography (TAT) \cite{Kunyansky2008,Hristova2008,Stefanov2009,Bal2011a}, photo-acoustic tomography (PAT) \cite{Kuchment2011a,Do2018}, magneto-acoustic tomography with magnetic induction (MAT-MI) \cite{Ammari2015,Qiu2015}, as well as MAET with voltage measurements \cite{Kunyansky2012}. 

When the sound speed is a known constant, as we assume, this inverse source problem is equivalent to the inversion of the spherical mean Radon transform, and has been intensively studied in the literature. Andersson \cite{Andersson1988} shows that the source $f$ can be inverted from measurements on a plane outside $\overline{\Omega}$ with a Lipschitz stability, using an explicit formula in terms of Fourier transform. An iterative formula is derived in \cite{Aramyan2024}. More studies consider measurements on a closed surface $\Sigma$ surrounding $\overline{\Omega}$. Explicit inversion formulas are given in \cite{Finch2004,Kunyansky2007a,Kunyansky2007} for spherical $\Sigma$, and in \cite{Haltmeier2014,Agranovsky2023} for more general convex $\Sigma$. A theoretically exact reconstruction from time-reduced measurements on a finite open surface is also developed in \cite{Do2018}. By the way, when the sound speed $c = c(x)$ varies spatially, the inverse source problem is also studied in \cite{Stefanov2009,Kuchment2011a}. Furthermore, if $c(x)$ is unknown, the joint recovery of $c(x)$ and the initial source $f$ is investigated in \cite{Liu2015a,Kian2025a,Kian2025}, where some uniqueness and stability results are derived for specifically structured $c(x)$.

In summary, with sufficiently many measurements generated by external point sources, we are able to determine the spatial distribution $f$ in effective ways mentioned above. Assuming that the mass density $\rho$ of the object is known, the vector-valued distribution $J_L\times\nu\delta_{\partial\Omega} + (\curl J_L)\chi_\Omega$ can be determined by setting the background magnetic field $B_0$ in three axial directions (actually, two directions are sufficient for the determination \cite{Kunyansky2012}). After that, the recovered distribution will be used in the electromagnetic part to reconstruct the conductivity $\sigma$ encoded in the vector field $J_L = \sigma(C_L - \nabla w_L)$.

\subsection{Recovering the adjoint current}

The information recovered in the first step includes the internal curl and the boundary tangential trace of the adjoint current field $J_L = \sigma(C_L - \nabla w_L)$. Since it satisfies the equation \eqref{eq:F}, we have the following explicit formula for recovering $J_L$ itself:
\begin{equation}\label{eq:vec_pot}
    J_L(x) = \curl \int_\Omega \nabla_x\times \frac{J_L(y)}{4\pi|x-y|}\,dy
    = \curl \left<J_L\times\nu\delta_{\partial\Omega} + (\curl J_L)\chi_\Omega, \frac{1}{4\pi|x-\cdot|}\right>.
\end{equation}
This formula is given in \cite{Stewart2011}, but for completeness we provide its proof here.
\begin{proof}[Proof of \eqref{eq:vec_pot}]
    Since $\frac{1}{4\pi|x|}$ is the fundamental solution of $-\Delta$ in $\R^3$, we have
    \begin{align*}
        J_L(x) &= \int_\Omega J_L(y)\delta(x-y)\,dy = -\int_\Omega \Delta_x \frac{J_L(y)}{4\pi|x-y|}\,dy \\
        &= \curl \int_\Omega \nabla_x\times \frac{J_L(y)}{4\pi|x-y|}\,dy - \nabla\int_\Omega \nabla_x\cdot \frac{J_L(y)}{4\pi|x-y|}\,dy.
    \end{align*}
    The second term vanishes since $J_L$ satisfies the equation \eqref{eq:F}, so that
    \begin{align*}
        \int_\Omega \nabla_x\cdot \frac{J_L(y)}{4\pi|x-y|}\,dy &= \int_\Omega J_L(y)\cdot\nabla_x \frac{1}{4\pi|x-y|}\,dy \\
        &= -\int_\Omega J_L(y)\cdot\nabla_y \frac{1}{4\pi|x-y|}\,dy \\
        &= -\int_{\partial\Omega} \frac{J_L(y)\cdot\nu(y)}{4\pi|x-y|}\,dS_y + \int_\Omega \frac{\div J_L(y)}{4\pi|x-y|}\,dy = 0.
    \end{align*}
    Hence, \eqref{eq:vec_pot} is proved.
\end{proof}

\subsection{Recovering the conductivity}

Finally, we consider recovering the conductivity $\sigma$ from the adjoint current $J_L = \sigma(C_L - \nabla w_L)$ in $\Omega$. To this end, we eliminate the gradient to derive
\begin{equation}\label{eq:curl}
    \curl\frac{J_L}{\sigma} = \curl C_L = -\frac{\mu}{4\pi}\left[\nabla\left(\div\int_D \frac{n(y)}{|y-x|}\,dS_y\right) - \Delta\int_D \frac{n(y)}{|y-x|}\,dS_y\right]
\end{equation}
for $x\in\Omega$. Since $\frac{1}{4\pi|x|}$ is the fundamental solution of $-\Delta$ in $\R^3$, we know that
\begin{equation*}
    -\Delta\int_D \frac{n(y)}{4\pi|y-x|}\,dS_y = 0
\end{equation*}
provided that $D\subset\R^3\setminus\overline{\Omega}$. Thus the equation \eqref{eq:curl} is equivalent to
\begin{equation}\label{eq:r}
    \curl(r J_L) = G_L \ \text{in}\ \Omega
\end{equation}
with $r\coloneqq 1/\sigma$ representing the electrical resistivity and
\begin{equation}\label{eq:GL}
    \begin{aligned}
        G_L(x) \coloneqq{}& -\frac{\mu}{4\pi}\nabla\left(\div\int_D \frac{n(y)}{|y-x|}\,dS_y\right) \\
        ={}& -\frac{\mu}{4\pi}\int_D \left(3\frac{(y-x)\cdot n(y)}{|y-x|^5}(y-x) - \frac{n(y)}{|y-x|^3}\right)\,dS_y.
    \end{aligned}
\end{equation}
At this stage, the unknown is the resistivity $r(x)$ which is a scalar function in $\Omega$. The vector field $J_L$ is known from the preceding reconstruction step, while $G_L$ is completely determined by the geometry of the measuring coil $L$. Therefore, for a fixed coil $L$, equation \eqref{eq:r} constitutes an overdetermined system of linear partial differential equations for the scalar unknown $r$. A natural way to reconstruct $r$ from the measured current $J_L$ is therefore to formulate the least-squares problem
\begin{equation}\label{eq:LS}
    \min_r \frac12 \|\curl(r J_L) - G_L\|_{L^2(\Omega)}^2,
\end{equation}
which can be solved using standard gradient-based iterative methods.

\begin{remark}
    In practice, the above least-squares formulation \eqref{eq:LS} should be interpreted together with the magnitude of the reconstructed adjoint current $J_L$. If $J_L$ vanishes, or is very small, in part of the domain, then the map $r\mapsto \curl(rJ_L)$ is degenerate or severely ill-conditioned there, and the resistivity $r$ cannot be expected to be stably reconstructed from this coil. Thus, after reconstructing $J_L$, one may choose a threshold $\tau>0$ and define a reconstruction region
    \begin{equation*}
        \Omega_\tau\coloneqq \{x\in\Omega: |J_L(x)|\ge \tau\}.
    \end{equation*}
    The data-fidelity integral in the least-squares functional may then be restricted to $\Omega_\tau$, and the complement $\Omega\setminus\Omega_\tau$ where $|J_L|<\tau$ is regarded as a low-sensitivity or poorly observable region for the fixed coil. Different coils may generate adjoint currents with different low-current regions, so multiple coils can be used to enlarge the reconstruction region.
\end{remark}

\begin{remark}
    The formulation \eqref{eq:LS} uses the data associated with a single measuring coil $L$, and the stability results below are also stated in this single-coil setting. If measurements from several coils $L_1,\cdots,L_N$ are available, one may instead combine the corresponding reconstructed currents and known coil fields by considering the joint least-squares functional
    \begin{equation*}
        \min_r \frac12\sum_{j=1}^N \|\nabla\times(rJ_{L_j}) - G_{L_j}\|^2_{L^2(\Omega)}.
    \end{equation*}
    Such a multi-coil formulation may improve conditioning and is natural for numerical implementation, while the theoretical analysis in the present work is restricted to the single-coil setting.
\end{remark}

In conclusion, we have designed an inversion procedure to reconstruct the conductivity distribution, which is shown in \cref{fig}.
\begin{figure}[htbp]
    \centering
    \begin{tikzpicture}[>=stealth]
        \matrix[column sep=8mm, row sep=6mm]{
            \node[draw, rounded corners] (data) {$\displaystyle m_L(t) = \frac{B_0}{\rho}\cdot \int_\Omega J_L(x)\times\nabla p(x,t)\,dx$}; & 
            \node[draw, rounded corners] (conduct) {$\displaystyle \phantom{\Big|}\sigma(x)\phantom{\Big|}$}; \\
            \node[align=center] (invI) {excitation by external point source \\ inverse source problem for acoustic wave equation}; &
            \node[] (invIII) {1st-order overdetermined linear system}; \\
            \node[draw] (curl) {$J_L\times\nu\delta_{\partial\Omega} + (\nabla\times J_L)\chi_\Omega$}; &
            \node[draw] (current) {$J_L = \sigma(C_L - \nabla w_L)$}; \\
        };

        \draw[-] (data) -- (invI);
        \draw[->] (invI) -- (curl);
        \draw[->] (curl) -- node[below] {explicit formula} (current);
        \draw[-] (current) -- (invIII);
        \draw[->] (invIII) -- (conduct);
    \end{tikzpicture}
    \caption{Diagram of inversion procedure}
    \label{fig}
\end{figure}
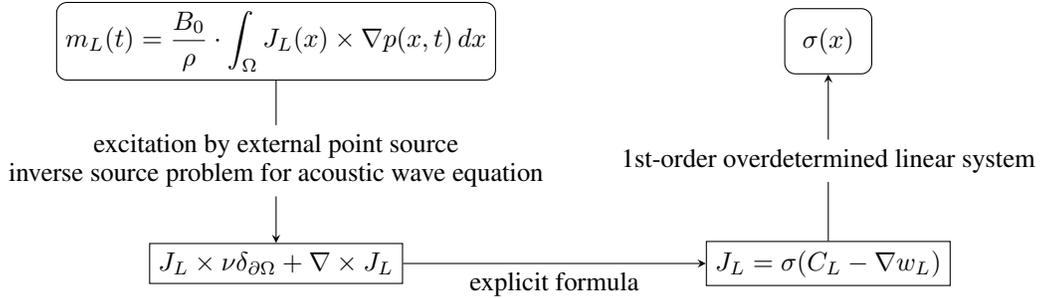

\section{Stability of recovering the conductivity from an internal current density}\label{sec:stab}

Within the inversion framework described above, the acoustic inverse source problem has been extensively investigated in the literature, and a variety of efficient reconstruction methods are available. Consequently, the remaining task is to analyze the stability of recovering the conductivity $\sigma$ from the measured current density field $J_L$.

The well-posedness of the forward problem has been established in \cref{prop:forward}. More precisely, for any fixed coil $L$ located outside the domain $\Omega$, the forward operator
\begin{equation}
    {\cal F}_L \colon \Sigma(\lambda,\Lambda) \to \Hc{\Omega},
    \quad
    \sigma \mapsto J_L = \sigma (C_L - \nabla w_L)
\end{equation}
is well-defined. In what follows, we fix a coil $L$ and investigate the stability of reconstructing $\sigma$ from the data ${\cal F}_L(\sigma)$.

\begin{proposition}\label{prop:weighted}
    Let $\Omega$ satisfy \cref{as} and $\sigma_1, \sigma_2\in\Sigma(\lambda,\Lambda)$. If $J_1 = {\cal F}_L(\sigma_1)$ and $J_2 = {\cal F}_L(\sigma_2)$, then
    \begin{equation}\label{eq:weighted}
        \int_\Omega |\sigma_1 - \sigma_2||J_1|^2\,dx \le C(\lambda,\Lambda,\Omega)\|C_L\|_{L^2(\Omega)}\|J_1 - J_2\|_{\Hc{\Omega}}.
    \end{equation}
\end{proposition}
\begin{proof}
    Let $r_1 = 1/\sigma_1$ and $r_2 = 1/\sigma_2$. Then
    \begin{equation*}
        \curl(r_1 J_1) = G_L = \curl(r_2 J_2) \ \text{in}\ \Omega.
    \end{equation*}
    It follows that
    \begin{equation*}
        \curl[(r_1-r_2)J_1] = \curl[r_2(J_2-J_1)] \ \text{in}\ \Omega.
    \end{equation*}
    Let $\sgn$ be the sign function on $\R$, that is,
	\begin{equation*}
		\sgn(t) = \begin{cases}
			1, & t>0, \\
			0, & t=0, \\
			-1, & t<0.
		\end{cases}
	\end{equation*}
	Then
	\begin{align}
		\curl(|r_1-r_2|J_1) &= \nabla|r_1-r_2|\times J_1 + |r_1-r_2|\curl J_1 \nonumber \\
		&= \sgn(r_1-r_2)[\nabla(r_1-r_2)\times J_1 + (r_1-r_2)\curl J_1] \nonumber \\
		&= \sgn(r_1-r_2)\curl[(r_1-r_2)J_1] \nonumber \\
		&= \sgn(r_1-r_2)\curl[r_2(J_2-J_1)]. \label{eq:id}
	\end{align}
    By \cref{cor:J_A}, there exists a vector field $A\in\Hc{\Omega}$ satisfying \eqref{eq:A} and \eqref{eq:est_A} such that $J_1 = \curl A$. Multiplying both sides of the equality \eqref{eq:id} by $A$ and integrating the left-hand side by parts, we have
    \begin{equation*}
        \int_\Omega |r_1-r_2||J_1|^2\,dx = \int_\Omega \sgn(r_1-r_2)A\cdot\curl[r_2(J_2-J_1)]\,dx,
    \end{equation*}
    where the boundary term vanishes since $A\times\nu = 0$ on $\partial\Omega$. It follows that
    \begin{align*}
        \int_\Omega |r_1-r_2||J_1|^2\,dx &\le \|A\|_{L^2(\Omega)} \|r_2\|_{W^{1,\infty}(\Omega)} \|J_1-J_2\|_{\Hc{\Omega}} \\
        &\le C(\lambda,\Lambda,\Omega)\|C_L\|_{L^2(\Omega)}\|J_1-J_2\|_{\Hc{\Omega}}.
    \end{align*}
    Noting that
    \begin{equation*}
        |r_1 - r_2| = \frac{|\sigma_1 - \sigma_2|}{\sigma_1\sigma_2} \ge \lambda^2|\sigma_1 - \sigma_2|,
    \end{equation*}
    the weighted estimate \eqref{eq:weighted} is proved.
\end{proof}

\begin{remark}
    We emphasize that no boundary condition of the form $\sigma_1 = \sigma_2$ on $\partial\Omega$ is required. 
    Indeed, one can select a vector potential $A$ associated with $J_1$ such that $A \times \nu = 0$ on $\partial\Omega$,
    which eliminates the boundary contribution arising from integration by parts.
    This observation highlights a structural advantage of the magnetic field measurement mode in MAET: 
    the internal conductivity can, in principle, be stably recovered without any a priori knowledge of its boundary values.
\end{remark}

This weighted estimate shows that $r$ can be uniquely determined where $J_1\ne 0$, and can be stably recovered in any subdomain where $|J_1|$ admits a positive lower bound. We emphasize that this condition in a subdomain should be interpreted as a local or region-of-interest statement: an external coil $L$ provides sufficient sensitivity in subdomains where the corresponding adjoint current is bounded away from zero.

\begin{example}
    If $D$ is in a plane $P$ outside $\Omega$, then its normal vector $n$ is constant. Suppose $P = \{y\in\R^3: y\cdot n = a\}$ with $a\in\R$. Then
    \begin{align}
        C_L(x) &= -\frac{\mu}{4\pi}\int_D \frac{y-x}{|y-x|^3}\,dS_y \times n, \label{eq:C_L2} \\
        G_L(x) &= -\frac{\mu}{4\pi}\left[\int_D \frac{3(y-x)\cdot n}{|y-x|^5}(y-x)\,dS_y - \left(\int_D \frac{1}{|y-x|^3}\,dS_y\right)n\right] \nonumber \\
        &= -\frac{\mu}{4\pi}\left[3(a - x\cdot n)\int_D \frac{y-x}{|y-x|^5}\,dS_y - \left(\int_D \frac{1}{|y-x|^3}\,dS_y\right)n\right]. \label{eq:G_L}
    \end{align}
    Here, note that $|a - x\cdot n| = \dist(x,P)$.
    
    Furthermore, if $D$ is a disk, say $D = B(y_0,r)\cap P$ with $r>0$ and $y_0\in P$:
    \begin{itemize}
        \item for $x\in y_0 + \R n$, which is the central axis of $D$: by symmetry we have
        \begin{equation*}
            \int_D \frac{y-x}{|y-x|^5}\,dS_y = (a - x\cdot n)\left(\int_D \frac{1}{|y-x|^5}\,dS_y\right)n,
        \end{equation*}
        so
        \begin{equation*}
            G_L(x) = -\frac{\mu}{4\pi}\left[3\dist(x,P)^2 \int_D \frac{1}{|y-x|^5}\,dS_y - \int_D \frac{1}{|y-x|^3}\,dS_y\right]n,
        \end{equation*}
        which is nonzero if
        \begin{equation*}
            \dist(x,P) > \frac{1}{\sqrt{3}}\sup_{y\in D} |y-x|.
        \end{equation*}
        In particular, $G_L\ne 0$ everywhere in $\Omega\cap(y_0+\R n)$ provided that $\dist(\Omega,P) > \diam(\Omega\cup D)/\sqrt{3}$;
        \item for $x\notin y_0 + \R n$: since $D$ has a symmetry plane across $x$, the integral
        \begin{equation*}
            \int_D \frac{y-x}{|y-x|^3}\,dS_y
        \end{equation*}
        must point to a direction that is not parallel to $n$, so $G_L(x)\ne 0$ noting that $|a - x\cdot n| = \dist(x,P) > 0$.
    \end{itemize}
    It shows that $J_L$ cannot vanish identically in any open subset of $\Omega$, otherwise it contradicts the equation \eqref{eq:r}. This observation rules out open regions of complete insensitivity, but it does not exclude isolated or lower-dimensional zero sets of $J_L$. This motivates the interior H\"older stability estimate in the following subsection, which does not require a pointwise positive lower bound for $J_L$.
\end{example}

\subsection{An interior H\"older stability}

In this subsection, we show that the equation \eqref{eq:r} prevents $J_L$ from vanishing too strongly on compact subsets of $\Omega$. Combined with the weighted estimate \eqref{eq:weighted}, this yields an interior H\"older stability estimate without imposing a pointwise positive lower bound on $J_L$. The related weighted interpolation argument has been used to derive conditional stability estimates in hybrid inverse problems in the presence of zeros of internal data \cite{Choulli2019a,Choulli2021,Bonnetier2022}. The distinctive feature of the present setting is that the required polynomial lower bound for the local $L^2$-norm of $J_L$ follows from the analyticity and intrinsic nontriviality of the coil-determined field $G_L$.

\begin{lemma}\label{lem:GL}
    Let $L\subset\R^3\setminus\overline{\Omega}$ be a nonempty $C^1$ oriented simple closed curve. Then the field $G_L$ defined by \eqref{eq:GL} is analytic in $\Omega$ and is not identically zero in any nonempty open subset of $\Omega$.
\end{lemma}
\begin{proof}
    Let $D\subset\R^3\setminus\overline{\Omega}$ be an oriented surface satisfying $\partial D = L$. Then the Newtonian potential $\int_D \frac{n(y)}{|y-x|}\,dS_y$ is harmonic in $\Omega$. Hence $G_L$ defined by \eqref{eq:GL} is also harmonic, and therefore analytic, in $\Omega$.

    We next prove the nonvanishing property of $G_L$. For $x\notin D$, Stokes' formula gives
    \begin{align*}
        C_L(x) &= -\frac{\mu}{4\pi}\int_D \frac{(y-x)\times n(y)}{|y-x|^3}\,dS_y \\
        &= -\frac{\mu}{4\pi}\int_D n(y)\times\nabla_y\frac{1}{|y-x|}\,dS_y = -\frac{\mu}{4\pi}\oint_L \frac{d\ell_y}{|y-x|},
    \end{align*}
    where $d\ell$ denotes the oriented line element along $L$. The line integral on the right-hand side is well-defined for every $x\in\R^3\setminus L$ and defines an analytic extension $\tilde{C}_L$ of $C_L$ to $\R^3\setminus L$. Set $\tilde{G}_L \coloneqq \curl\tilde{C}_L$, which is analytic in $\R^3\setminus L$ and agrees with $G_L$ in $\Omega$. 
    
    Suppose that $G_L$ vanishes identically in some nonempty open subset of $\Omega$. Then analyticity forces $\tilde{G}_L\equiv 0$ in $\R^3\setminus L$. This implies that $\tilde{G}_L = 0$ as a distribution in $\R^3$, since
    \begin{equation*}
        \tilde{G}_L(x) = \curl\tilde{C}_L(x) = -\frac{\mu}{4\pi}\oint_L \frac{(y-x)\times d\ell_y}{|y-x|^3}
    \end{equation*}
    is locally integrable in $\R^3$. On the other hand, denoting $\Gamma(x) = \frac{1}{4\pi|x|}$ and $\ell\delta_L$ as a vector-valued distribution of line integral on $L$, we have $\tilde{C}_L = -\mu\Gamma*\ell\delta_L$. Since $L$ is closed, $\div(\ell\delta_L) = 0$, and hence $\div\tilde{C}_L = 0$. It follows that
    \begin{equation*}
        \curl\tilde{G}_L = -\Delta\tilde{C}_L = -\mu(-\Delta\Gamma)*\ell\delta_L = -\mu\ell\delta_L
    \end{equation*}
    as a nonzero distribution in $\R^3$. This contradicts $\tilde{G}_L = 0$ in $\R^3$.
\end{proof}

\begin{proposition}\label{prop:lb}
    Let $L\subset\R^3\setminus\overline{\Omega}$ be a nonempty $C^1$ oriented simple closed curve, $\sigma\in\Sigma(\lambda,\Lambda)$ and $J_L = {\cal F}_L(\sigma)$. Then for any compact set $K\subset\subset\Omega$, there exist constants $c, \theta, \rho_0>0$, depending on $\lambda,\Omega,K$ and the coil $L$, such that
    \begin{equation}\label{eq:lb}
        \int_{B(x_0,\rho)} |J_L(x)|^2\,dx \ge c\rho^\theta
    \end{equation}
    for all $x_0\in K$ and $\rho\in(0,\rho_0)$.
\end{proposition}

\begin{proof}
    By \cref{lem:GL}, $G_L$ is analytic in $\Omega$ and is not identically zero in any nonempty open subset of $\Omega$. It leads to a finite-order nondegeneracy property that is uniform on $K$. Indeed, for every $x\in K$, there exists a multi-index $\beta_x$ such that $\partial^{\beta_x} G_L(x)\ne 0$. By smoothness of $G_L$ and compactness of $K$, there exists an integer $M\ge 0$ and a constant $c_0>0$ such that
    \begin{equation}\label{eq:order}
        \max_{|\beta|\le M} |\partial^\beta G_L(x)| \ge c_0, \quad \forall\,x\in K.
    \end{equation}

    For every multi-index $\beta$ with $|\beta|\le M$, choose
    $\psi_\beta\in C_c^\infty(B(0,1))$ satisfying
    \begin{equation}\label{eq:moments}
        \int_{B(0,1)} z^\gamma\psi_\beta(z)\,dz = \beta!\,\delta_{\beta\gamma}, \quad |\gamma|\le M.
    \end{equation}
    Such functions can be constructed by taking $\psi_\beta = \eta q_\beta$, where $\eta\in C_c^\infty(B(0,1))$ satisfies $0\le\eta\le 1$ and $\eta\equiv 1$ in $B(0,1/2)$, and $q_\beta$ is a polynomial of degree at most $M$.
    The coefficients of $q_\beta$ are determined by
    \eqref{eq:moments} since the associated Gram matrix
    \begin{equation*}
        \left(\int_{B(0,1)}z^{\beta+\gamma}\eta(z)\,dz\right)_{|\beta|,|\gamma|\le M}
    \end{equation*}
    is positive definite and hence invertible.

    Now fix arbitrary $x_0\in K$. By \eqref{eq:order}, there exists a multi-index $\beta$ with $m\coloneqq|\beta|\le M$, such that $|\partial^\beta G_L(x_0)| \ge c_0$. For any $\rho\in(0,\dist(K,\partial\Omega)/2)$, define
    \begin{equation*}
        e \coloneqq \frac{\partial^\beta G_L(x_0)}{|\partial^\beta G_L(x_0)|}, \quad
        \Phi_\rho(x) \coloneqq \psi_\beta\left(\frac{x-x_0}{\rho}\right)e
    \end{equation*}
    Then the Taylor expansion of $G_L$ at $x_0$, together with \eqref{eq:moments} gives 
    \begin{align*}
        \int_{B(x_0,\rho)} G_L\cdot\Phi_\rho\,dx
        &= \rho^3 \int_{B(0,1)} G_L(x_0+\rho z)\cdot e\, \psi_\beta(z)\,dz \\
        &= \rho^3 \int_{B(0,1)} \Big[\partial^\beta G_L(x_0)(\rho z)^\beta + O(\rho^{M+1})\Big] \cdot e\,\psi_\beta(z)\,dz \\
        &= \rho^{m+3} |\partial^\beta G_L(x_0)| + O(\rho^{M+4}).
    \end{align*}
    Here the remainder term admits a uniform estimate depending only on $\Omega,K,L$ since $G_L\in C^{M+1}(\overline{\Omega})$ and there are only finitely many functions $\psi_\beta$. Consequently, there exists a constant $\rho_0 = C(\Omega,K,L)\in (0,1]$ such that
    \begin{equation}\label{eq:ge}
        \int_{B(x_0,\rho)} G_L\cdot\Phi_\rho\,dx \ge \frac{c_0}{2}\rho^{m+3}, \quad \forall\,\rho\in(0,\rho_0).
    \end{equation}
    Now by the equation \eqref{eq:r},
    \begin{align}
        \int_{B(x_0,\rho)} G_L\cdot\Phi_\rho\,dx &= \int_{B(x_0,\rho)} r J_L\cdot\curl\Phi_\rho\,dx \notag \\
        &\le \lambda^{-1} \|J_L\|_{L^2(B(x_0,\rho))} \|\curl\Phi_\rho\|_{L^2(B(x_0,\rho))} \notag \\
        &\le C(\lambda,\Omega,K,L) \rho^{1/2} \|J_L\|_{L^2(B(x_0,\rho))}. \label{eq:le}
    \end{align}
    Since $\rho_0\le 1$ and $m\le M$, combining \eqref{eq:ge} and \eqref{eq:le} proves \eqref{eq:lb} with $\theta = 2M+5$.
\end{proof}

\begin{remark}
    If $G_L$ does not vanish on $K$, then $\inf_K |G_L| > 0$. In this case we can take $M=0$ in the proof above, and hence $\theta=5$.
\end{remark}

With this polynomial lower bound estimate for the local $L^2$-norm of $J_L$, we derive the following interior H\"older stability for the inverse problem by a weighted interpolation inequality.
\begin{theorem}
    Let $\Omega$ satisfy \cref{as} and $L\subset\R^3\setminus\overline{\Omega}$ be a nonempty $C^1$ oriented simple closed curve. Suppose $\sigma_1, \sigma_2\in\Sigma(\lambda,\Lambda)$ and $J_1 = {\cal F}_L(\sigma_1), J_2 = {\cal F}_L(\sigma_2)$. Then for any compact set $K\subset\subset\Omega$, there exist constants $C>0$ and $\alpha\in(0,1)$, depending on $\lambda,\Lambda,\Omega,K$ and the coil $L$, such that
    \begin{equation}\label{eq:Holder}
        \|\sigma_1 - \sigma_2\|_{L^\infty(K)} \le C\|J_1 - J_2\|_{H({\rm curl},\Omega)}^\alpha.
    \end{equation}
\end{theorem}
\begin{proof}
    Set $\sigma = \sigma_1-\sigma_2$. For any compact set $K\subset\subset\Omega$, take $x_0\in K$ such that $|\sigma(x_0)| = \|\sigma\|_{L^\infty(K)}$. For any $\rho>0$ with $B(x_0,\rho)\subset\Omega$, since $\|\sigma\|_{W^{1,\infty}(\Omega)} \le 2\Lambda$, we have
	\begin{equation*}
		|\sigma(x_0)| \le |\sigma(x)| + 2\Lambda\rho, \quad \forall\,x\in B(x_0,\rho).
	\end{equation*}
	Therefore, by the weighted stability estimate \eqref{eq:weighted} we have
    \begin{align}
        |\sigma(x_0)|\int_{B(x_0,\rho)} |J_1|^2\,dx &\le \int_{B(x_0,\rho)} |\sigma||J_1|^2\,dx + 2\Lambda\rho\int_{B(x_0,\rho)} |J_1|^2\,dx \notag \\
        &\le \int_\Omega |\sigma||J_1|^2\,dx + 2\Lambda\rho\int_{B(x_0,\rho)} |J_1|^2\,dx \notag \\
		&\le C(\lambda,\Lambda,\Omega,L)\|J_1 - J_2\|_{H({\rm curl},\Omega)} + 2\Lambda\rho\int_{B(x_0,\rho)} |J_1|^2\,dx. \label{eq:est}
    \end{align}
	According to \cref{prop:lb}, we know that there exist constants $c,\theta,\rho_0>0$, depending on $\lambda,\Omega,K,L$, such that
	\begin{equation*}
		\int_{B(x_0,\rho)} |J_1|^2\,dx \ge c\rho^\theta, \quad \forall\,\rho\in(0,\rho_0).
	\end{equation*}
	So it follows from \eqref{eq:est} that
	\begin{align*}
		\|\sigma\|_{L^\infty(K)} = |\sigma(x_0)| &\le \frac{C(\lambda,\Lambda,\Omega)\|J_1 - J_2\|_{H({\rm curl},\Omega)}}{\int_{B(x_0,\rho)} |J_1|^2\,dx} + 2\Lambda\rho \\
		&\le C_0(\lambda,\Lambda,\Omega,K,L)\|J_1 - J_2\|_{H({\rm curl},\Omega)}\rho^{-\theta} + 2\Lambda\rho.
	\end{align*}
    Taking
	\begin{equation*}
		\rho = \min\left\{\left(\frac{2\Lambda}{\theta C_0\|J_1 - J_2\|_{H({\rm curl},\Omega)}}\right)^{-1/(1+\theta)}, \rho_0\right\},
	\end{equation*}
	leads to \eqref{eq:Holder} with $\alpha = 1/(1+\theta)$.
\end{proof}

\subsection{Lipschitz stability under a pointwise nonzero constraint}

The interior H\"older stability estimate \eqref{eq:Holder} requires no pointwise positive lower bound on the adjoint current $J_L$. The price is a weaker stability modulus, whose exponent depends on the possible vanishing order of the coil-determined field $G_L$. We now turn to the stronger region-of-interest Lipschitz stability available under a pointwise nonzero constraint.

\begin{assumption}\label{as:CL}
    There exists a constant $c>0$ and a subdomain $\Omega'$ of $\Omega$, such that $|C_L|\ge c$ in $\Omega'$.
\end{assumption}

Note that $C_L$ is completely determined by the coil geometry and can be computed explicitly from the definition \eqref{eq:C_L}. Thus, for any prescribed coil $L$, one may identify the subdomain where $|C_L|$ is bounded away from zero and take it as the candidate reconstruction region $\Omega'$. In the following result, we introduce a smallness condition ensuring that the correction term $\nabla w_L$, generated by the adjoint problem, cannot substantially cancel $C_L$ in $\Omega'$.

\begin{theorem}\label{thm:lb_J}
    Let $\Omega$ satisfy \cref{as} with $\partial\Omega\in C^{1,\alpha}$ for an $\alpha\in(0,1)$ and $q = \frac{3}{1-\alpha}$. If $\sigma\in\Sigma(\lambda,\Lambda)$ and \cref{as:CL} holds, then there exists a constant $C_1$ depending on $\lambda,\Lambda,\alpha$ and $\Omega$, such that if
    \begin{equation}\label{eq:cond}
        \|\div(\sigma C_L)\|_{L^q(\Omega)} + \|C_L\cdot\nu\|_{C^{0,\alpha}(\partial\Omega)} \le \frac{c}{2C_1},
    \end{equation}
    then $J_L = \sigma(C_L - \nabla w_L)$ with $w_L$ solving \eqref{eq:w} satisfies
    \begin{equation*}
        |J_L|\ge \frac{\lambda c}{2} \ \text{in}\ \Omega'.
    \end{equation*}
\end{theorem}

\begin{proof}
    Without loss of generality, we assume that $\int_\Omega w_L\,dx = 0$, since by \cref{prop:forward}, the solution $w_L$ is unique up to an additive constant. Then by Poincar\'e's inequality \eqref{eq:Poincare} and testing the equation \eqref{eq:w} with $w_L$ itself we have
    \begin{equation*}
        \|w_L\|_{L^2(\Omega)} \le C(\Omega)\|\nabla w_L\|_{L^2(\Omega)}\le C(\lambda,\Omega)\Big(\|\div(\sigma C_L)\|_{L^2(\Omega)} + \|C_L\cdot\nu\|_{L^2(\partial\Omega)}\Big).
    \end{equation*}
    Combining Theorem 5.31 and Corollary 5.32 in \cite{Lieberman2013}, we derive
    \begin{align*}
        \|w_L\|_{L^\infty(\Omega)} &\le C(\lambda,\alpha,\Omega)\Big(\|w_L\|_{L^2(\Omega)} + \|\div(\sigma C_L)\|_{L^q(\Omega)} + \|C_L\cdot\nu\|_{L^\infty(\partial\Omega)}\Big) \\
        &\le C(\lambda,\alpha,\Omega)\Big(\|\div(\sigma C_L)\|_{L^q(\Omega)} + \|C_L\cdot\nu\|_{L^\infty(\partial\Omega)}\Big).
    \end{align*}
    By the Schauder estimate \cite[Theorem 5.54]{Lieberman2013}
    \begin{equation*}
        \|w_L\|_{C^{1,\alpha}(\overline{\Omega})} \le C(\lambda,\Lambda,\Omega)\Big(\|w_L\|_{L^\infty(\Omega)} + \|\div(\sigma C_L)\|_{L^q(\Omega)} + \|C_L\cdot\nu\|_{C^{0,\alpha}(\partial\Omega)}\Big),
    \end{equation*}
    we have
    \begin{equation*}
        \|\nabla w_L\|_{L^\infty(\Omega)} \le C_1(\lambda,\Lambda,\alpha,\Omega)\Big(\|\div(\sigma C_L)\|_{L^q(\Omega)} + \|C_L\cdot\nu\|_{C^{0,\alpha}(\partial\Omega)}\Big).
    \end{equation*}
    Therefore, if \eqref{eq:cond} holds with the constant $C_1$ above, we have
    \begin{equation*}
        \|\nabla w_L\|_{L^\infty(\Omega)} \le \frac{c}{2},
    \end{equation*}
    so that
    \begin{equation*}
        |J_L| \ge \lambda|C_L - \nabla w_L| \ge \frac{\lambda c}{2} \ \text{in}\ \Omega'.
    \end{equation*}
\end{proof}

\begin{remark}
    The smallness condition \eqref{eq:cond} actually contains two requirements:
    \begin{enumerate}
        \item Since $\div C_L = 0$, we have $\div(\sigma C_L) = \nabla\sigma\cdot C_L$, so the bound requires that $C_L$ should be ``almost orthogonal'' to the deviation of $\sigma$. For example, if $\sigma$ varies only in one direction $e$, then we can set the coil $L$ in a plane with normal $e$, so that $\nabla\sigma\cdot C_L = 0$ since $C_L\perp e$.
        
        \item The second part requires that $C_L$ is ``almost tangential'' on $\partial\Omega$. For example, if $\Omega$ is rotationally symmetric with an axis $\R e$, then we can set the coil $L$ as a circle with the same axis, so that $C_L\cdot\nu = 0$ on $\partial\Omega$ since $C_L$ circulates around this axis; see \cref{fig:ex}. However, such $C_L$ vanishes on the axis, so we can only guarantee the non-vanishing property of $J_L$ in $\Omega'$ that is away from the axis.
    \end{enumerate}
\end{remark}

\begin{figure}[htbp]
    \centering
    \tdplotsetmaincoords{90}{105}
    \begin{tikzpicture}[tdplot_main_coords, scale=0.8, ->-/.style={
        postaction={decorate}, decoration={
            markings,
            mark=at position 0.3 with {\arrow{>}}
        }
    }, >=stealth]
        \draw[dashed] (0,-8,0) -- (0,8,0);
        \foreach \phi in {105,106,...,283} {
            \pgfmathsetmacro\tou{10+(\phi-105)*5/18}
            \fill[black!\tou, opacity=0.5] 
                plot[smooth, domain=0:180, samples=30, variable=\theta] 
                ({3*sin(\theta)*cos(\phi)}, {4*sin(\theta)*sin(\phi)-2}, {3*cos(\theta)}) --
                plot[smooth, domain=180:0, samples=30, variable=\theta] 
                ({3*sin(\theta)*cos(\phi+2)}, {4*sin(\theta)*sin(\phi+2)-2}, {3*cos(\theta)}) --
                cycle;
        }
        \foreach \phi in {105,285} {
            \draw
                plot[smooth, domain=0:180, samples=30, variable=\theta] 
                ({3*sin(\theta)*cos(\phi)}, {4*sin(\theta)*sin(\phi)-2}, {3*cos(\theta)});
        }
        \node at (0,1,1) {$\Omega$};
        \fill[black!10, opacity=0.5]
            plot[smooth, domain=0:360, samples=60, variable=\theta]
            ({1.4*cos(\theta)}, 4.5, {1.4*sin(\theta)});
        \node at (0,4.5,0.7) {$D$};
        \draw[->]
            plot[smooth, domain=360:0, samples=60, variable=\theta]
            ({1.4*cos(\theta)}, 4.5, {1.4*sin(\theta)});
        \node at (0,3.7,-0.3) {$L$};
        \draw[->] (0,4.5,0) -- (0,5.5,0);
        \node at (0,5.8,0.4) {$n = e$};
        \pgfmathsetmacro{\ra}{1.5*sqrt(3)}
        \draw[->-] 
            plot[smooth, domain=-90:90, samples=30, variable=\theta]
            ({\ra*cos(\theta)}, -4, {\ra*sin(\theta)});
        \draw[->-, densely dashed] 
            plot[smooth, domain=90:270, samples=30, variable=\theta]
            ({\ra*cos(\theta)}, -4, {\ra*sin(\theta)});
        \node at (0,-4.2,-1) {$C_L$};
        \pgfmathsetmacro{\rb}{0.75*sqrt(15)}
        \draw[->-] 
            plot[smooth, domain=-90:90, samples=30, variable=\theta]
            ({\rb*cos(\theta)}, -1, {\rb*sin(\theta)});
        \draw[->-, densely dashed] 
            plot[smooth, domain=90:270, samples=30, variable=\theta]
            ({\rb*cos(\theta)}, -1, {\rb*sin(\theta)});
        \node at (0,-1.26,-1.1) {$C_L$};
    \end{tikzpicture}
    \caption{Rotationally symmetric domain and coil with the same axis}
    \label{fig:ex}
\end{figure}
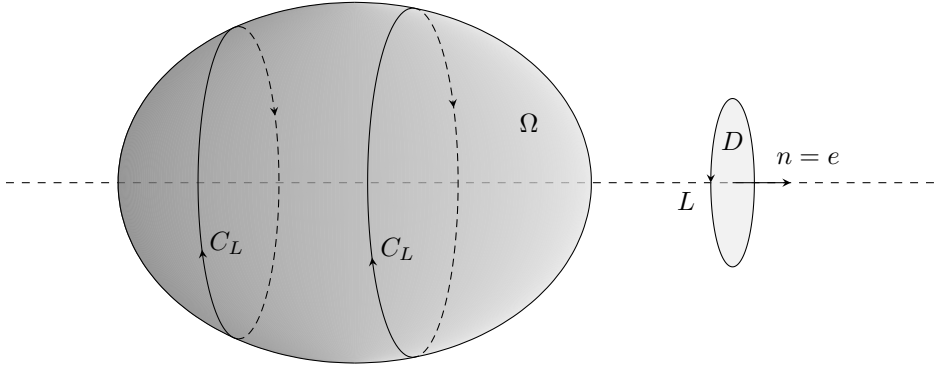

Through the weighted estimate \eqref{eq:weighted}, we directly obtain the following conditional stability result.
\begin{corollary}
    Let $\Omega$ satisfy \cref{as} with $\partial\Omega\in C^{1,\alpha}$ for an $\alpha\in(0,1)$ and $q = \frac{3}{1-\alpha}$. Suppose \cref{as:CL} holds. If $\sigma_1,\sigma_2\in\Sigma(\lambda,\Lambda)$, then there exists a constant $C_1$ depending on $\lambda,\Lambda,\alpha,\Omega$ such that if
    \begin{equation*}
        \|\nabla\sigma_1\cdot C_L\|_{L^q(\Omega)} + \|C_L\cdot\nu\|_{C^{0,\alpha}(\partial\Omega)} \le \frac{c}{2C_1},
    \end{equation*}
    then
    \begin{equation}
        \|\sigma_1 - \sigma_2\|_{L^1(\Omega')} \le C_{\rm S}\|C_L\|_{L^2(\Omega)} \|{\cal F}_L(\sigma_1) - {\cal F}_L(\sigma_2)\|_{\Hc{\Omega}}
    \end{equation}
    for some constant $C_{\rm S}>0$ depending on $\lambda,\Lambda,\alpha,\Omega$ and $c$.
\end{corollary}

\subsection{Nonzero constraint and stability in half-space}

In this subsection, we consider when the imaging domain $\Omega$ is regarded as a half-space, say $\R^3_+\coloneqq\R^2\times(0,\infty)$, which is a reasonable approximation of the scenario that the coil $L$ is placed near a flat surface of a biological tissue. This half-space geometry should be understood as an idealized full-aperture setting. More precisely, we assume that point-source acoustic excitations can be generated for all source locations on the entire external plane. Under this full-aperture assumption, the acoustic inverse source problem in the first reconstruction step is treated as stably solvable. The sources producing the prescribed static magnetic field $B_0$, as well as the measuring coil $L$, are assumed to be located outside the imaging region. Under these settings, the adjoint current density $J_L$ is regarded as already reconstructed, and the purpose of this subsection is to analyze the stability of recovering the conductivity $\sigma$ from $J_L$. We do not address here the limited-aperture case, where measurements are available only on a bounded surface and only visible singularities of the acoustic source can be expected to be reconstructed stably.

Setting the measuring coil $L$ in an exterior plane parallel to the boundary $\partial\Omega = \R^2\times\{0\}$, we know from \eqref{eq:C_L2} that $C_L\cdot\nu = 0$ on $\partial\Omega$, so that the smallness condition \eqref{eq:cond} contains only the source term.
However, we cannot directly use the estimates above to derive analogous results, since $\Omega$ is unbounded in this case. For simplicity, we make the following assumptions; see \cref{fig:half-space}.
\begin{assumption}\label{as:sigma}
    Let $\sigma\colon \R^3_+\to\R_+$ satisfy
    \begin{itemize}
        \item $\sigma\equiv\sigma_0$ outside a half-ball $B_+(R)\coloneqq\R^3_+\cap\{|x|<R\}$ with $R\ge 1$;
        \item $\sigma\in\Sigma(\lambda,\Lambda)$ with $\Omega$ replaced by $B_+(R)$.
    \end{itemize}
    Noting that $\sigma\in C^0(\overline{B_+(R)})$ by Sobolev embedding, we further assume that 
    \begin{itemize}
        \item $\sigma$ is continuous across the interface $\partial B_+(R)\cap\R^3_+$.
    \end{itemize}
\end{assumption}

\begin{figure}[htbp]
    \centering
    \tdplotsetmaincoords{100}{80}
    \begin{tikzpicture}[tdplot_main_coords, scale=1.5]
        \fill[gray!20, opacity=0.7] (-2,-2,0) -- (2,-2,0) -- (2,2,0) -- (-2,2,0) -- cycle;
        \foreach \phi in {80,81,...,258} {
            \pgfmathsetmacro\tou{10+(260-\phi)*5/18}
            \fill[black!\tou, opacity=0.5] 
                plot[smooth, domain=0:90, samples=15, variable=\theta] 
                ({sin(\theta)*cos(\phi)}, {sin(\theta)*sin(\phi)}, {cos(\theta)}) --
                plot[smooth, domain=90:0, samples=15, variable=\theta] 
                ({sin(\theta)*cos(\phi+2)}, {sin(\theta)*sin(\phi+2)}, {cos(\theta)}) --
                cycle;
        }
        \foreach \phi in {80,260} {
            \draw
                plot[smooth, domain=0:90, samples=30, variable=\theta] 
                ({sin(\theta)*cos(\phi)}, {sin(\theta)*sin(\phi)}, {cos(\theta)});
        }
        \draw
            plot[smooth, domain=80:260, samples=30, variable=\phi] 
            ({cos(\phi)}, {sin(\phi)}, 0);
        \draw[dashed] 
            plot[smooth, domain=-100:80, samples=30, variable=\phi] 
            ({cos(\phi)}, {sin(\phi)}, 0);
        \node at (0,0,-0.7) {$\sigma\equiv 0$};
        \node at (0,1.5,1) {$\sigma\equiv\sigma_0$};
        \node at (1,0,0.3) {$B_+(R)$};
        \node at (-1.2,-1.5,0) {$\R^2\times\{0\}$};
    \end{tikzpicture}
    \caption{Assumptions of conductivity distribution in half-space}
    \label{fig:half-space}
\end{figure}
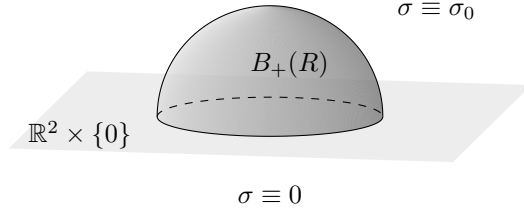

With this assumption, the solution of \eqref{eq:w} can be expressed using a Green function of the operator $-\div(\sigma\nabla)$ with Neumann boundary condition in the half-space $\R^3_+$:
\begin{equation}\label{eq:Green}
    w_L(x) = -\int_{B_+(R)} G(x,y)f(y)\,dy
\end{equation}
with $f = \nabla\sigma\cdot C_L$, which vanishes outside $B_+(R)$ due to the first assumption on $\sigma$. The existence of $G$ and its well-known global pointwise estimate
\begin{equation}\label{eq:G}
    0 \le G(x,y) \le C(\lambda)|x-y|^{-1}, \quad x\ne y,
\end{equation}
can be demonstrated as in \cite{Hofmann2007} and \cite{Kang2010} for Dirichlet problem in unbounded domain. Using the Green function representation \eqref{eq:Green} and the pointwise estimate \eqref{eq:G}, we first derive that for any $x\in B_+(2R)$,
\begin{align}
    |w_L(x)| &\le C(\lambda)\int_{B_+(R)} \frac{|f(y)|}{|x-y|}\,dy \notag \\
    &\le C(\lambda)\|f\|_{L^\infty(B_+(R))} R^2. \label{eq:est_0}
\end{align}
Next, we proceed as in \cite[Lemma 3.1]{Gruter1982} to derive the pointwise estimate of $\nabla w_L$ in $B_+(R)$:
\begin{lemma}
    Let $\sigma\in\Sigma(\lambda,\Lambda)$ with $\Omega=\R^3_+$. Suppose $v$ is a bounded solution to 
    \begin{equation*}
        \begin{cases}
            \div(\sigma\nabla v) = f & \text{in}\ \R^3_+, \\
            \partial_\nu v = 0 & \text{on}\ \partial\R^3_+
        \end{cases}
    \end{equation*}
    with $f\in L^\infty(\R^3_+)$. Then for any $R\ge 1$ and any $x\in B_+(2R)$,
    \begin{equation}\label{eq:est1}
        |\nabla v(x)| \le C(\lambda,\Lambda)\frac{R}{2R-|x|}\left(\|f\|_{L^\infty(B_+(2R))} + \|v\|_{L^\infty(B_+(2R))}\right).
    \end{equation}
    In particular,
    \begin{equation}\label{eq:est_1}
        \|\nabla v\|_{L^\infty(B_+(R))} \le C(\lambda,\Lambda)\left(\|f\|_{L^\infty(B_+(2R))} + \|v\|_{L^\infty(B_+(2R))}\right).
    \end{equation}
\end{lemma}
\begin{proof}
    Extend $\sigma,v$ and $f$ evenly across $\partial\R^3_+$. The homogeneous Neumann condition implies that the extensions, denoted by the same symbols, satisfy $\div(\sigma\nabla v) = f$ weakly in $\R^3$. By standard interior regularity, $v\in C^1(\overline{B(0,2R)})$. Set
    \begin{equation*}
        M_0 \coloneqq \sup_{B(0,2R)}|v|, \quad M_1 \coloneqq \sup_{B(0,2R)} (2R-|x|)|\nabla v(x)|.
    \end{equation*}
    If $M_1 = 0$, there is nothing to prove. Otherwise, there exists an $x_0\in B_+(2R)$ such that $(2R-|x_0|)|\nabla v(x_0)| = M_1 > 0$. 
    For any $r\le R-\frac{1}{2}|x_0|$, define a cut-off function $\eta\in C_c^\infty(B(x_0,r))$ satisfying
    \begin{equation*}
        0\le\eta\le 1, \quad \eta = 1\ \text{in}\ B(x_0,r/2), \quad |\nabla\eta|\le \frac{C}{r}, \quad |\Delta\eta|\le \frac{C}{r^2}.
    \end{equation*}
    Denote the fundamental solution of $-\sigma(x_0)\Delta$ by $G_0(x,y) = \sigma(x_0)^{-1}\Gamma(x-y)$.
    Applying $\eta G_0(\cdot,y)$ with arbitrary $y\in B_r\coloneqq B(x_0,r)$ as a test function leads to
    \begin{align}
        -\int_{B_r} f\eta G_0\,dx ={}& \int_{B_r} \sigma\nabla v\cdot\nabla_x(\eta G_0)\,dx \notag \\
        ={}& \sigma(x_0)\int_{B_r} \nabla v\cdot(\eta\nabla_x G_0 + G_0\nabla\eta)\,dx \label{eq:test} \\
        &+ \int_{B_r} (\sigma - \sigma(x_0)) \nabla v \cdot(\eta\nabla_x G_0 + G_0\nabla\eta)\,dx. \notag
    \end{align}
    Integrating the first term by parts, it becomes
    \begin{align*}
        & \sigma(x_0)\int_{B_r} \nabla v\cdot(\eta\nabla_x G_0 + G_0\nabla\eta)\,dx \\
        ={}& \sigma(x_0)\int_{B_r} [\nabla(\eta v) - v\nabla\eta]\cdot\nabla_x G_0\,dx - \sigma(x_0)\int_{B_r} v \nabla_x\cdot(G_0\nabla\eta)\,dx \\
        ={}& \eta(y)v(y) - \sigma(x_0)\int_{B_r} v(2\nabla\eta\cdot\nabla_x G_0 + G_0\Delta\eta)\,dx.
    \end{align*}
    Therefore, for $i=1,2,3$, differentiating \eqref{eq:test} with respect to $y_i$ and letting $y = x_0$ we obtain
    \begin{align*}
        \partial_i v(x_0) ={}& -\int_{B_r} f\eta\,\partial_{y_i}G_0(x,x_0)\,dx \\
        &+ \sigma(x_0)\int_{B_r} v[2\nabla\eta\cdot\partial_{y_i}\nabla_x G_0(x,x_0) + \Delta\eta\,\partial_{y_i} G_0(x,x_0)]\,dx \\
        &- \int_{B_r} (\sigma - \sigma(x_0)) \nabla v\cdot[\eta\,\partial_{y_i}\nabla_x G_0(x,x_0) + \partial_{y_i}G_0(x,x_0)\nabla\eta]\,dx.
    \end{align*}
    Since direct computations on $G_0$ show that
    \begin{equation*}
        |\nabla_y G_0(x,y)|\le C(\lambda)|x-y|^{-2}, \quad |\nabla_y\nabla_x G_0(x,y)|\le C(\lambda)|x-y|^{-3}
    \end{equation*}
    for all $x,y\in\R^3$ with $x\ne y$, we derive that
    \begin{align*}
        |\nabla v(x_0)| \le{}& C(\lambda)\bigg(
            \int_{B_r} \frac{|f(x)|}{|x-x_0|^2}\,dx 
            + \frac{M_0}{r} \int_{B_r\setminus B_{r/2}} \frac{1}{|x-x_0|^3}\,dx \\
            &+ \frac{M_0}{r^2} \int_{B_r\setminus B_{r/2}} \frac{1}{|x-x_0|^2}\,dx
        \bigg) \\
        &+ \frac{C(\lambda,\Lambda)M_1}{2R-|x_0|-r} \bigg(
            \int_{B_r} \frac{1}{|x-x_0|^2}\,dx 
            + \frac{1}{r} \int_{B_r\setminus B_{r/2}} \frac{1}{|x-x_0|}\,dx
        \bigg) \\
        \le{}& C(\lambda)\left(\|f\|_{L^\infty(B(0,2R))} r + \frac{M_0}{r}\right) + \frac{C_0(\lambda,\Lambda)M_1 r}{2R-|x_0|-r}.
    \end{align*}
    Noting that $|\nabla v(x_0)| = \frac{M_1}{2R-|x_0|}$, we take
    \begin{equation*}
        r = \min\left\{R-\frac{1}{2}|x_0|,\ \frac{2R-|x_0|}{1+2C_0(2R-|x_0|)}\right\}
    \end{equation*}
    so that the last term can be absorbed. Since $R\ge 1$ and $r\le \frac{1}{2C_0}$, we obtain
    \begin{align*}
        M_1 &\le C(\lambda,\Lambda)\left(\frac{2R-|x_0|}{2C_0}\|f\|_{L^\infty(B(0,2R))} + M_0\max\{2,\ 1+2C_0(2R-|x_0|)\}\right) \\
        &\le C(\lambda,\Lambda)R \big(\|f\|_{L^\infty(B(0,2R))} + M_0\big).
    \end{align*}
    Restricting the even extensions to $\R^3_+$ proves \eqref{eq:est1}, and \eqref{eq:est_1} follows from $2R-|x|\ge R$ for $x\in B_+(R)$.
\end{proof}

Combining \eqref{eq:est_1} with the estimate \eqref{eq:est_0}, we obtain
\begin{equation}\label{eq:C1}
    \sup_{B_+(R)} |\nabla w_L| \le C_1(\lambda,\Lambda) R^2 \|\nabla\sigma\cdot C_L\|_{L^\infty(B_+(R))}.
\end{equation}
In conclusion, we have
\begin{theorem}\label{thm:stab_ball}
    Let $\sigma_1,\sigma_2$ satisfy \cref{as:sigma} and $J_1 = {\cal F}_L(\sigma_1)$, $J_2 = {\cal F}_L(\sigma_2)$. Suppose
    \begin{itemize}
        \item $L\subset\R^2\times\{-a\}$ with $a>0$, and
        \item $|C_L|\ge c$ for a constant $c>0$ in a subdomain $\Omega'$ of $B_+(R)$.
    \end{itemize}
    Then there exists a constant $C_1$ depending on $\lambda,\Lambda$, such that if
    \begin{equation}\label{eq:small}
        \|\nabla\sigma_1\cdot C_L\|_{L^\infty(B_+(R))} \le \frac{c}{2C_1 R^2},
    \end{equation}
    then $|J_1| \ge \lambda c/2$ in $\Omega'$, and
    \begin{equation}\label{eq:stab}
        \|\sigma_1 - \sigma_2\|_{L^1(\Omega')} \le C_{\rm S} R^{5/2} \|J_1 - J_2\|_{\Hc{B_+(R)}}
    \end{equation}
    for some constant $C_{\rm S}>0$ depending on $\lambda,\Lambda,a,c$ and the area bounded by $L$.
\end{theorem}

\begin{remark}
    The parameter $R$ is the radius of the region within which the conductivity $\sigma$ is inhomogeneous and the source $\nabla\sigma\cdot C_L$ is supported. It impacts the stability estimate in two ways: a larger $R$ results in a more restrictive smallness condition \eqref{eq:small} on the unknown conductivity, as well as a larger stability constant $C_{\rm S}R^{5/2}$.
\end{remark}

The lower bound of $|J_1|$ comes directly from \eqref{eq:C1} and \eqref{eq:small}. To derive the stability estimate \eqref{eq:stab}, we still cannot directly use the previous weighted estimate \eqref{eq:weighted}, since it involves the application of \cref{lem:vec_pot}. This result of normal vector potential only applies for bounded domains and for vector fields that are tangential on the whole boundary. Now if we tend to find a vector potential $A$ of $J_1$ in $B_+(R)$, since $J_1$ is tangential only on the flat boundary $\Gamma_0\coloneqq\partial B_+(R)\cap\partial\R^3_+$, we can only construct $A$ that is normal on $\Gamma_0$. The idea is to extend $J_1$ to a slightly larger half-ball so that it is tangential on the whole boundary, and we can apply \cref{lem:vec_pot}.

\begin{lemma}\label{lem:A_ball}
    For any vector field $F\in\Hd{B_+(R)}$ satisfying
    \begin{equation}\label{eq:F_ball}
        \begin{cases}
            \div F = 0 & \text{in}\ B_+(R), \\ 
            F\cdot\nu = 0 & \text{on}\ \Gamma_0,
        \end{cases}
    \end{equation}
    there exists a vector field $A\in\Hc{B_+(R)}$ satisfying
    \begin{equation}\label{eq:A_ball}
        \begin{cases}
            \div A = 0 & \text{in}\ B_+(R), \\
            A\times\nu = 0 & \text{on}\ \Gamma_0
        \end{cases}
    \end{equation}
    and a generic constant $C>0$, such that $F = \curl A$ in $B_+(R)$ and
    \begin{equation}\label{eq:A_est}
        \|A\|_{L^2(B_+(R))} \le CR\|F\|_{L^2(B_+(R))}.
    \end{equation}
\end{lemma}
\begin{proof}
    By a scaling argument, we only need to prove this result for $R = 1$. Denote
    \begin{equation*}
        D\coloneqq B_+(2)\setminus\overline{B_+(1)}, \quad \Gamma\coloneqq\partial B_+(1)\cap\R^3_+.
    \end{equation*}
    Let $\psi\in H^1(D)$ be a weak solution of the Neumann problem
    \begin{equation*}
        \begin{cases}
            \Delta\psi = 0 & \text{in}\ D, \\
            \partial_\nu\psi = F\cdot\nu & \text{on}\ \Gamma, \\
            \partial_\nu\psi = 0 & \text{on}\ \partial D\setminus\Gamma.
        \end{cases}
    \end{equation*}
    It is unique up to an additive constant, so we take the one satisfying $\int_D \psi\,dx = 0$. Then by the trace theorem \eqref{eq:tr} and Poincar\'e's inequality \eqref{eq:Poincare}, we have
    \begin{equation*}
        \|\nabla\psi\|_{L^2(D)}^2 = \langle F\cdot\nu, \psi\rangle_{\Gamma}\le \|\psi\|_{H^{1/2}(\Gamma)}\|F\cdot\nu\|_{H^{-1/2}(\Gamma)} \le C\|\nabla\psi\|_{L^2(D)}\|F\cdot\nu\|_{H^{-1/2}(\Gamma)}.
    \end{equation*}
    Therefore,
    \begin{equation*}
        \|\nabla\psi\|_{L^2(D)} \le C\|F\cdot\nu\|_{H^{-1/2}(\Gamma)} \le C\|F\|_{L^2(B_+(1))},
    \end{equation*}
    where the last inequality is due to the estimate \eqref{eq:nor_tr} of the normal trace and the fact that $F$ is solenoidal in $B_+(1)$. Now let
    \begin{equation*}
        \tilde{F} = \begin{cases}
            F & \text{in}\ B_+(1), \\
            \nabla\psi & \text{in}\ D.
        \end{cases}
    \end{equation*}
    Since $\tilde{F}\in L^2(B_+(2))$ satisfying
    \begin{equation*}
        \begin{cases}
            \div \tilde{F} = 0 & \text{in}\ B_+(2), \\ 
            \tilde{F}\cdot\nu = 0 & \text{on}\ \partial B_+(2),
        \end{cases}
    \end{equation*}
    by \cref{lem:vec_pot} it admits a vector potential $\tilde{A}\in\Hc{B_+(2)}$ satisfying
    \begin{equation*}
        \begin{cases}
            \div \tilde{A} = 0 & \text{in}\ B_+(2), \\
            \tilde{A}\times\nu = 0 & \text{on}\ \partial B_+(2)
        \end{cases}
    \end{equation*}
    and
    \begin{equation*}
        \|\tilde{A}\|_{L^2(B_+(2))} \le C\|\tilde{F}\|_{L^2(B_+(2))} \le C\|F\|_{L^2(B_+(1))}.
    \end{equation*}
    Then $A\coloneqq\tilde{A}|_{B_+(1)}$ is a vector potential of $F$ and satisfies \eqref{eq:A_ball} and \eqref{eq:A_est} with $R = 1$.
\end{proof}

\begin{proof}[Proof of \cref{thm:stab_ball}]
    It remains to prove \eqref{eq:stab}. Analogously to the proof of \cref{prop:weighted}, by applying \cref{lem:A_ball} to $J_1$ instead of using \cref{cor:J_A} we have
    \begin{align}
        \int_{B_+(R)} |\sigma_1 - \sigma_2||J_1|^2\,dx &\le C(\lambda,\Lambda) \|A\|_{L^2(B_+(R))}\|J_1 - J_2\|_{\Hc{B_+(R)}} \nonumber \\
        &\le C(\lambda,\Lambda) R \|J_1\|_{L^2(B_+(R))}\|J_1 - J_2\|_{\Hc{B_+(R)}}. \label{eq:weighted_ball}
    \end{align}
    Now we estimate the $L^2$ norm of $J_1$. By Caccioppoli's inequality, the weak solution $w_1$ of \eqref{eq:w} with $\sigma = \sigma_1$ satisfies the estimate
    \begin{equation*}
        \|\nabla w_1\|_{L^2(B_+(R))} \le C(\lambda)\left(R\|\nabla\sigma_1\cdot C_L\|_{L^2(B_+(R))} + \frac{1}{R}\|w_1\|_{L^2(B_+(2R))}\right).
    \end{equation*}
    Using the Green function representation \eqref{eq:Green} and the pointwise estimate \eqref{eq:G} we derive
    \begin{equation*}
        \|w_1\|_{L^2(B_+(2R))} \le C(\lambda) R^2 \|\nabla\sigma_1\cdot C_L\|_{L^2(B_+(R))}.
    \end{equation*}
    So 
    \begin{equation*}
        \|\nabla w_1\|_{L^2(B_+(R))} \le C(\lambda) R \|\nabla\sigma_1\cdot C_L\|_{L^2(B_+(R))}
    \end{equation*}
    and by the smallness condition \eqref{eq:small}, we have
    \begin{equation*}
        \|\nabla w_1\|_{L^2(B_+(R))} \le C(\lambda,\Lambda,c) R^{1/2}.
    \end{equation*}
    Moreover, according to \eqref{eq:C_L}, we know that $|C_L| \le Ca^{-2}\Area D$, leading to the estimate
    \begin{equation*}
        \|C_L\|_{L^2(B_+(R))} \le C(a,\Area D) R^{3/2}.
    \end{equation*}
    Thus,
    \begin{equation}\label{eq:J1_bound}
        \|J_1\|_{L^2(B_+(R))} \le \lambda^{-1}\Big(\|C_L\|_{L^2(B_+(R))} + \|\nabla w_1\|_{L^2(B_+(R))}\Big) \le C(\lambda,\Lambda,a,c,\Area D) R^{3/2}.
    \end{equation}
    Finally, \eqref{eq:stab} follows from \eqref{eq:weighted_ball}, \eqref{eq:J1_bound} and the previous result that $|J_1|\ge \lambda c/2$ in $\Omega'\subset B_+(R)$.
\end{proof}

\section{Concluding remarks}\label{sec:con}

In conclusion, this work provides a rigorous mathematical foundation for MAET with magnetic field measurements. We formulate a quasi-static model that accurately captures the underlying physics and introduce an adjoint problem to decouple the hybrid inverse problem. This decoupling yields a well-studied acoustic inverse source problem and a critical inverse conductivity problem with internal data. Our primary theoretical contribution lies in addressing the latter, namely recovering the conductivity from an internal adjoint current density. By a weighted interpolation argument and leveraging the analytic properties of the coil-generated fields, we establish an interior H\"older stability result without imposing a pointwise nonzero constraint on the internal data. Under a stronger nonzero constraint ensured by suitable smallness conditions on the conductivity and coil configuration, we further derive Lipschitz stability in regions of interest, both in bounded domains and in the half-space setting.

While the analysis identifies mechanisms that can lead to stable reconstruction, it does not provide a complete practical criterion for optimal coil design. Understanding how these conditions affect reconstruction quality in realistic configurations, and how multiple coils may improve stability and effectiveness, remains an important direction for future work. In particular, numerical studies on coil configuration, including position, orientation and number, would be valuable for assessing the practical performance of the proposed reconstruction framework.

\bibliographystyle{abbrv}
\bibliography{ref}

@book{Adams2003,
  title = {Sobolev {{Spaces}}},
  author = {Adams, Robert A. and Fournier, John J. F.},
  year = 2003,
  series = {Pure and Applied Mathematics},
  edition = {2nd},
  volume = {140},
  publisher = {Academic Press},
  address = {Oxford},
  langid = {english}
}

@book{Evans2010,
  title = {Partial Differential Equations},
  author = {Evans, Lawrence C.},
  year = 2010,
  series = {Graduate Studies in Mathematics},
  edition = {2nd},
  volume = {19},
  publisher = {American Mathematical Society},
  address = {Providence, Rhode Island},
  langid = {english},
  lccn = {515.353}
}

@article{Gruter1982,
  title = {The {{Green}} Function for Uniformly Elliptic Equations},
  author = {Gr{\"u}ter, Michael and Widman, Kjell-Ove},
  year = 1982,
  journal = {manuscripta mathematica},
  volume = {37},
  number = {3},
  pages = {303--342},
  doi = {10.1007/BF01166225},
  urldate = {2022-06-23},
  langid = {english}
}

@article{Hofmann2007,
  title = {The {{Green}} Function Estimates for Strongly Elliptic Systems of Second Order},
  author = {Hofmann, Steve and Kim, Seick},
  year = 2007,
  journal = {manuscripta mathematica},
  volume = {124},
  number = {2},
  pages = {139--172},
  doi = {10.1007/s00229-007-0107-1},
  urldate = {2022-03-30},
  langid = {english}
}

@article{Kang2010,
  title = {Global Pointwise Estimates for {{Green}}'s Matrix of Second Order Elliptic Systems},
  author = {Kang, Kyungkeun and Kim, Seick},
  year = 2010,
  journal = {Journal of Differential Equations},
  volume = {249},
  number = {11},
  pages = {2643--2662},
  doi = {10.1016/j.jde.2010.05.017},
  urldate = {2025-03-28},
  copyright = {https://www.elsevier.com/tdm/userlicense/1.0/},
  langid = {english}
}

@book{Lieberman2013,
  title = {Oblique Derivative Problems for Elliptic Equations},
  author = {Lieberman, Gary M.},
  year = 2013,
  publisher = {World Scientific},
  address = {Singapore},
  doi = {10.1142/8679},
  lccn = {QA377 .L528 2013}
}

@book{Alberti2018,
  title = {{Lectures on Elliptic Methods for Hybrid Inverse Problems}},
  author = {Alberti, Giovanni S. and Capdeboscq, Yves},
  year = 2018,
  series = {{Collection SMF Cours sp\'ecialis\'es}},
  volume = {25},
  publisher = {Soci\'et\'e math\'ematique de France},
  address = {Paris},
  langid = {fre eng}
}

@article{Alberti2022a,
  title = {Non-Zero Constraints in Elliptic {{PDE}} with Random Boundary Values and Applications to Hybrid Inverse Problems},
  author = {Alberti, Giovanni S.},
  year = 2022,
  journal = {Inverse Problems},
  volume = {38},
  number = {12},
  pages = {124005},
  doi = {10.1088/1361-6420/ac9924},
  urldate = {2022-12-28},
  langid = {english}
}

@incollection{Bal2013a,
  title = {Hybrid Inverse Problems and Internal Functionals},
  booktitle = {Inverse Problems and Applications: Inside out {{II}}},
  author = {Bal, Guillaume},
  editor = {Uhlmann, Gunther},
  year = 2013,
  series = {Mathematical {{Sciences Research Institute Publications}}},
  volume = {60},
  pages = {325--368},
  publisher = {Cambridge University Press},
  address = {Cambridge},
  langid = {english}
}

@article{Choulli2021,
  title = {Some Stability Inequalities for Hybrid Inverse Problems},
  author = {Choulli, Mourad},
  year = 2021,
  journal = {Comptes Rendus. Math\'ematique},
  volume = {359},
  number = {10},
  pages = {1251--1265},
  doi = {10.5802/crmath.262},
  urldate = {2025-04-02},
  langid = {english}
}

@article{Widlak2012,
  title = {Hybrid Tomography for Conductivity Imaging},
  author = {Widlak, Thomas and Scherzer, Otmar},
  year = 2012,
  journal = {Inverse Problems},
  volume = {28},
  number = {8},
  pages = {084008},
  doi = {10.1088/0266-5611/28/8/084008},
  urldate = {2022-03-30},
  langid = {english}
}

@article{Choulli2019a,
  title = {H\"older Stability for an Inverse Medium Problem with Internal Data},
  author = {Choulli, Mourad and Triki, Faouzi},
  year = 2019,
  journal = {Research in the Mathematical Sciences},
  volume = {6},
  number = {1},
  publisher = {Springer},
  doi = {10.1007/s40687-018-0171-z},
  urldate = {2025-10-31},
  chapter = {9},
  langid = {english}
}

@article{Uhlmann2009,
  title = {Electrical Impedance Tomography and {{Calder\'on}}'s Problem},
  author = {Uhlmann, Gunther},
  year = 2009,
  journal = {Inverse Problems},
  volume = {25},
  number = {12},
  pages = {123011},
  doi = {10.1088/0266-5611/25/12/123011},
  urldate = {2022-03-30},
  langid = {english}
}

@article{Haider2008,
  title = {Magneto-Acousto-Electrical Tomography: A Potential Method for Imaging Current Density and Electrical Impedance},
  shorttitle = {Magneto-Acousto-Electrical Tomography},
  author = {Haider, Saja and Hrbek, A. and Xu, Yuan},
  year = 2008,
  journal = {Physiological Measurement},
  volume = {29},
  number = {6},
  pages = {S41-S50},
  doi = {10.1088/0967-3334/29/6/S04},
  urldate = {2022-03-30},
  langid = {english}
}

@article{Kunyansky2012,
  title = {A Mathematical Model and Inversion Procedure for Magneto-Acousto-Electric Tomography},
  author = {Kunyansky, Leonid},
  year = 2012,
  journal = {Inverse Problems},
  volume = {28},
  number = {3},
  pages = {035002},
  doi = {10.1088/0266-5611/28/3/035002},
  urldate = {2022-03-30},
  langid = {english}
}

@article{Kunyansky2017,
  title = {Rotational Magneto-Acousto-Electric Tomography ({{MAET}}): Theory and Experimental Validation},
  shorttitle = {Rotational Magneto-Acousto-Electric Tomography ({{MAET}})},
  author = {Kunyansky, Leonid and Ingram, Charles P. and Witte, Russell S.},
  year = 2017,
  journal = {Physics in Medicine \& Biology},
  volume = {62},
  number = {8},
  pages = {3025--3050},
  doi = {10.1088/1361-6560/aa6222},
  urldate = {2022-03-30},
  langid = {english}
}

@article{Kunyansky2023,
  title = {Weighted {{Radon}} Transforms of Vector Fields, with Applications to Magnetoacoustoelectric Tomography},
  author = {Kunyansky, Leonid and McDugald, Edward and Shearer, Benjamin},
  year = 2023,
  journal = {Inverse Problems},
  volume = {39},
  number = {6},
  pages = {065014},
  doi = {10.1088/1361-6420/acd07a},
  urldate = {2023-05-22},
  langid = {english}
}

@article{Roth2009,
  title = {Ultrasonically-Induced {{Lorentz}} Force Tomography},
  author = {Roth, Bradley J. and Schalte, Kevin},
  year = 2009,
  journal = {Medical \& Biological Engineering \& Computing},
  volume = {47},
  number = {6},
  pages = {573--577},
  doi = {10.1007/s11517-009-0476-6},
  urldate = {2022-03-30},
  langid = {english}
}

@article{Stewart2011,
  title = {Longitudinal and Transverse Components of a Vector Field},
  author = {Stewart, Andrew M.},
  year = 2011,
  journal = {Sri Lankan Journal of Physics},
  volume = {12},
  eprint = {0801.0335},
  primaryclass = {physics},
  pages = {33--42},
  doi = {10.4038/sljp.v12i0.3504},
  urldate = {2024-04-02},
  archiveprefix = {arXiv}
}

@article{Wen1998,
  title = {Hall Effect Imaging},
  author = {Wen, Han and Shah, Jatin and Balaban, Robert S.},
  year = 1998,
  journal = {IEEE Transactions on Biomedical Engineering},
  volume = {45},
  number = {1},
  pages = {119--124},
  doi = {10.1109/10.650364},
  urldate = {2022-03-30},
  langid = {english}
}

@article{Gozu2024,
  title = {Analyzing Pulse Compression Performance and Image Quality Metrics of Different Excitations in {{MAET}} with Magnetic Field Measurements},
  author = {G{\"o}z{\"u}, Mehmet Soner and Gen{\c c}er, Nevzat G{\"u}neri},
  year = 2024,
  journal = {International Journal for Numerical Methods in Biomedical Engineering},
  volume = {40},
  number = {12},
  pages = {e3890},
  doi = {10.1002/cnm.3890},
  urldate = {2024-12-19},
  copyright = {\copyright{} 2024 John Wiley \& Sons Ltd.},
  langid = {english}
}

@article{Guo2015,
  title = {Magneto-Acousto-Electrical Tomography with Magnetic Induction for Conductivity Reconstruction},
  author = {Guo, Liang and Liu, Guoqiang and Xia, Hui},
  year = 2015,
  journal = {IEEE Transactions on Biomedical Engineering},
  volume = {62},
  number = {9},
  pages = {2114--2124},
  doi = {10.1109/TBME.2014.2382562},
  urldate = {2022-03-30},
  langid = {english}
}

@article{Kaboutari2019,
  title = {Data Acquisition System for {{MAET}} with Magnetic Field Measurements},
  author = {Kaboutari, Keivan and Tetik, Ahmet {\"O}nder and Ghalichi, Elyar and G{\"o}z{\"u}, Mehmet Soner and Zengin, Reyhan and Gen{\c c}er, Nevzat G{\"u}neri},
  year = 2019,
  journal = {Physics in Medicine \& Biology},
  volume = {64},
  number = {11},
  pages = {115016},
  publisher = {IOP Publishing},
  doi = {10.1088/1361-6560/ab1809},
  urldate = {2024-12-19},
  langid = {english}
}

@inproceedings{Tetik2023,
  title = {{{MAET}} with Magnetic Field Measurements Using Circular and Figure-of-Eight Coils},
  booktitle = {2023 {{IEEE Biomedical Circuits}} and {{Systems Conference}} ({{BioCAS}})},
  author = {Tetik, Ahmet {\"O}nder and Gen{\c c}er, Nevzat G{\"u}neri},
  year = 2023,
  pages = {1--5},
  doi = {10.1109/BioCAS58349.2023.10389086},
  urldate = {2024-12-19}
}

@phdthesis{Zengin2012,
  title = {Electrical Impedance Tomography Using {{Lorentz}} Fields},
  author = {Zengin, Reyhan},
  year = 2012,
  urldate = {2023-05-23},
  langid = {english},
  school = {Middle East Technical University}
}

@article{Zengin2016,
  title = {Lorentz Force Electrical Impedance Tomography Using Magnetic Field Measurements},
  author = {Zengin, Reyhan and Gen{\c c}er, Nevzat G{\"u}neri},
  year = 2016,
  journal = {Physics in Medicine \& Biology},
  volume = {61},
  number = {16},
  pages = {5887--5905},
  doi = {10.1088/0031-9155/61/16/5887},
  urldate = {2022-03-30},
  langid = {english}
}

@article{Amrouche1998,
  title = {Vector Potentials in Three-Dimensional Non-Smooth Domains},
  author = {Amrouche, Ch{\'e}rif and Bernardi, Christine and Dauge, Monique and Girault, Vivette},
  year = 1998,
  journal = {Mathematical Methods in the Applied Sciences},
  volume = {21},
  number = {9},
  pages = {823--864},
  doi = {10.1002/(SICI)1099-1476(199806)21:9<823::AID-MMA976>3.0.CO;2-B},
  urldate = {2023-07-30},
  langid = {english}
}

@article{Amrouche2013,
  title = {Lp-Theory for Vector Potentials and {{Sobolev}}'s Inequalities for Vector Fields: Application to the {{Stokes}} Equations with Pressure Boundary Conditions},
  author = {Amrouche, Ch{\'e}rif and Seloula, Nour El Houda},
  year = 2013,
  journal = {Mathematical Models and Methods in Applied Sciences},
  volume = {23},
  number = {01},
  pages = {37--92},
  publisher = {World Scientific Publishing Co.},
  doi = {10.1142/S0218202512500455},
  urldate = {2026-01-05}
}

@article{Ammari2015,
  title = {A Mathematical and Numerical Framework for Magnetoacoustic Tomography with Magnetic Induction},
  author = {Ammari, Habib and Boulmier, Simon and Millien, Pierre},
  year = 2015,
  journal = {Journal of Differential Equations},
  volume = {259},
  number = {10},
  pages = {5379--5405},
  doi = {10.1016/j.jde.2015.06.040},
  urldate = {2022-03-30},
  langid = {english}
}

@article{Qiu2015,
  title = {Analysis of the Magnetoacoustic Tomography with Magnetic Induction},
  author = {Qiu, Lingyun and Santosa, Fadil},
  year = 2015,
  journal = {SIAM Journal on Imaging Sciences},
  volume = {8},
  number = {3},
  pages = {2070--2086},
  doi = {10.1137/15M1012323},
  urldate = {2022-03-30},
  langid = {english}
}

@article{Bonnetier2022,
  title = {Stability for Quantitative Photoacoustic Tomography Revisited},
  author = {Bonnetier, Eric and Choulli, Mourad and Triki, Faouzi},
  year = 2022,
  journal = {Research in the Mathematical Sciences},
  volume = {9},
  number = {2},
  doi = {10.1007/s40687-022-00322-6},
  urldate = {2022-05-27},
  chapter = {24},
  langid = {english}
}

@article{Agranovsky2023,
  title = {On the Exactness of the Universal Backprojection Formula for the Spherical Means {{Radon}} Transform},
  author = {Agranovsky, Mark and Kunyansky, Leonid},
  year = 2023,
  journal = {Inverse Problems},
  volume = {39},
  number = {3},
  pages = {035002},
  doi = {10.1088/1361-6420/acb2ee}
}

@article{Andersson1988,
  title = {On the Determination of a Function from Spherical Averages},
  author = {Andersson, Lars-Erik},
  year = 1988,
  journal = {SIAM Journal on Mathematical Analysis},
  volume = {19},
  number = {1},
  pages = {214--232},
  publisher = {{Society for Industrial and Applied Mathematics}},
  doi = {10.1137/0519016},
  urldate = {2025-12-10}
}

@article{Aramyan2024,
  title = {Recovering a Function from Spherical Means in {{3D}} Using Local Data},
  author = {Aramyan, Rafik},
  year = 2024,
  journal = {Inverse Problems and Imaging},
  volume = {18},
  number = {3},
  pages = {690--707},
  doi = {10.3934/ipi.2023050},
  urldate = {2025-12-10}
}

@article{Bal2011a,
  title = {Quantitative Thermo-Acoustics and Related Problems},
  author = {Bal, Guillaume and Ren, Kui and Uhlmann, Gunther and Zhou, Ting},
  year = 2011,
  journal = {Inverse Problems},
  volume = {27},
  number = {5},
  pages = {055007},
  publisher = {IOP Publishing},
  doi = {10.1088/0266-5611/27/5/055007},
  urldate = {2022-07-08},
  langid = {english}
}

@article{Do2018,
  title = {Theoretically Exact Photoacoustic Reconstruction from Spatially and Temporally Reduced Data},
  author = {Do, Ngoc and Kunyansky, Leonid},
  year = 2018,
  journal = {Inverse Problems},
  volume = {34},
  number = {9},
  pages = {094004},
  doi = {10.1088/1361-6420/aacfac},
  urldate = {2022-03-30},
  langid = {english}
}

@article{Finch2004,
  title = {Determining a Function from Its Mean Values over a Family of Spheres},
  author = {Finch, David and Patch, Sarah K. and {Rakesh}},
  year = 2004,
  journal = {SIAM Journal on Mathematical Analysis},
  volume = {35},
  number = {5},
  pages = {1213--1240},
  doi = {10.1137/S0036141002417814},
  urldate = {2022-03-30},
  langid = {english}
}

@article{Haltmeier2014,
  title = {Universal Inversion Formulas for Recovering a Function from Spherical Means},
  author = {Haltmeier, Markus},
  year = 2014,
  journal = {SIAM Journal on Mathematical Analysis},
  volume = {46},
  number = {1},
  pages = {214--232},
  doi = {10.1137/120881270},
  urldate = {2022-03-30},
  langid = {english}
}

@article{Hristova2008,
  title = {Reconstruction and Time Reversal in Thermoacoustic Tomography in Acoustically Homogeneous and Inhomogeneous Media},
  author = {Hristova, Yulia and Kuchment, Peter and Nguyen, Linh},
  year = 2008,
  journal = {Inverse Problems},
  volume = {24},
  number = {5},
  pages = {055006},
  doi = {10.1088/0266-5611/24/5/055006},
  urldate = {2022-03-30},
  langid = {english}
}

@incollection{Kuchment2011a,
  title = {Mathematics of Photoacoustic and Thermoacoustic Tomography},
  booktitle = {Handbook of {{Mathematical Methods}} in {{Imaging}}},
  author = {Kuchment, Peter and Kunyansky, Leonid},
  editor = {Scherzer, Otmar},
  year = 2011,
  pages = {817--865},
  publisher = {Springer},
  address = {New York},
  doi = {10.1007/978-0-387-92920-0_19},
  urldate = {2022-11-24},
  langid = {english}
}

@article{Kunyansky2007,
  title = {A Series Solution and a Fast Algorithm for the Inversion of the Spherical Mean {{Radon}} Transform},
  author = {Kunyansky, Leonid},
  year = 2007,
  journal = {Inverse Problems},
  volume = {23},
  number = {6},
  pages = {S11-S20},
  doi = {10.1088/0266-5611/23/6/S02},
  urldate = {2022-03-30},
  langid = {english}
}

@article{Kunyansky2007a,
  title = {Explicit Inversion Formulae for the Spherical Mean {{Radon}} Transform},
  author = {Kunyansky, Leonid},
  year = 2007,
  journal = {Inverse Problems},
  volume = {23},
  number = {1},
  pages = {373--383},
  doi = {10.1088/0266-5611/23/1/021},
  urldate = {2022-03-30},
  langid = {english}
}

@article{Kunyansky2008,
  title = {Thermoacoustic Tomography with Detectors on an Open Curve: An Efficient Reconstruction Algorithm},
  shorttitle = {Thermoacoustic Tomography with Detectors on an Open Curve},
  author = {Kunyansky, Leonid},
  year = 2008,
  journal = {Inverse Problems},
  volume = {24},
  number = {5},
  pages = {055021},
  doi = {10.1088/0266-5611/24/5/055021},
  urldate = {2022-03-30},
  langid = {english}
}

@article{Stefanov2009,
  title = {Thermoacoustic Tomography with Variable Sound Speed},
  author = {Stefanov, Plamen and Uhlmann, Gunther},
  year = 2009,
  journal = {Inverse Problems},
  volume = {25},
  number = {7},
  pages = {075011},
  doi = {10.1088/0266-5611/25/7/075011},
  urldate = {2022-03-30},
  langid = {english}
}

@misc{Kian2025,
  title = {Uniqueness and Stability in Determining the Wave Equation from a Single Passive Boundary Measurement},
  author = {Kian, Yavar and Liu, Hongyu},
  year = 2025,
  number = {arXiv:2507.10012},
  eprint = {2507.10012},
  primaryclass = {math},
  publisher = {arXiv},
  doi = {10.48550/arXiv.2507.10012},
  urldate = {2025-08-04},
  archiveprefix = {arXiv}
}

@article{Kian2025a,
  title = {Determination of the Sound Speed and an Initial Source in Photoacoustic Tomography},
  author = {Kian, Yavar and Uhlmann, Gunther},
  year = 2025,
  journal = {Transactions of the American Mathematical Society},
  volume = {378},
  number = {8},
  pages = {5329--5353},
  doi = {10.1090/tran/9467},
  urldate = {2025-06-03},
  copyright = {https://www.ams.org/publications/copyright-and-permissions},
  langid = {english}
}

@article{Liu2015a,
  title = {Determining Both Sound Speed and Internal Source in Thermo- and Photo-Acoustic Tomography},
  author = {Liu, Hongyu and Uhlmann, Gunther},
  year = 2015,
  journal = {Inverse Problems},
  volume = {31},
  number = {10},
  pages = {105005},
  doi = {10.1088/0266-5611/31/10/105005},
  urldate = {2025-08-04},
  langid = {english}
}

@article{Qiu2023,
  title = {Lipschitz Stability of Recovering the Conductivity from Internal Current Densities},
  author = {Qiu, Lingyun and Zheng, Siqin},
  year = 2023,
  journal = {Inverse Problems},
  volume = {39},
  number = {8},
  pages = {085010},
  publisher = {IOP Publishing},
  doi = {10.1088/1361-6420/acdea0},
  urldate = {2023-07-01},
  copyright = {All rights reserved},
  langid = {english}
}

@article{Gencer1999,
  title = {Electrical Conductivity Imaging via Contactless Measurements},
  author = {Gen{\c c}er, Nevzat G{\"u}neri and Tek, M. Nejat},
  year = 1999,
  journal = {IEEE Transactions on Medical Imaging},
  volume = {18},
  number = {7},
  pages = {617--627},
  doi = {10.1109/42.790461},
  urldate = {2022-03-30},
  langid = {english}
}

@article{Tarjan1968,
  title = {Electrodeless Measurements of the Effective Resistivity of the Human Torso and Head by Magnetic Induction},
  author = {Tarjan, Peter P. and McFee, Richard},
  year = 1968,
  journal = {IEEE Transactions on Biomedical Engineering},
  volume = {BME-15},
  number = {4},
  pages = {266--278},
  doi = {10.1109/TBME.1968.4502577},
  urldate = {2026-02-13}
}

\end{document}